\documentclass[aap,MSNbibl,seceqn,dvips]{arximspdf}
\usepackage{mathbh}
\usepackage{graphicx}

\doi{10.1214/13-AAP942} 
\volume{24}
\issue{3}
\pubyear{2014}
\firstpage{1081}
\lastpage{1128}

\makeatletter
\newcommand{\rrvert}{\vert}
\newcommand{\llvert}{\vert}
\newcommand{\eqref}[1]{(\ref{#1})}
\newtheorem{theorem}{Theorem}[section]
\newtheorem{corollary}[theorem]{Corollary}
\newtheorem{lemma}[theorem]{Lemma}

\newproclaim{remark}[theorem]{Remark}
\newproclaim{conjecture}[theorem]{Conjecture}
\newcommand{\ind}{\mathbh{1}}
\newcommand{\parrow}{\mathop{\longrightarrow}\limits^{p_{\nu}}_{n \to\infty}}

\newcommand{\asarrows}{\mathop{\rightarrow}\limits^{\mathrm{a.s.}}}
\newcommand{\Poisson}{\mathcal{P}}

\makeatother

\begin{document}
\begin{frontmatter}

\title{Epidemics on random intersection graphs}
\runtitle{Epidemics on random intersection graphs}

\begin{aug}
\author[1]{\fnms{Frank G.} \snm{Ball}\thanksref{t1}\ead[label=e1]{Frank.Ball@nottingham.ac.uk}},
\author[2]{\fnms{David J.} \snm{Sirl}\corref{}\thanksref{t1}\ead[label=e2]{D.Sirl@lboro.ac.uk}}
\and
\author[3]{\fnms{Pieter} \snm{Trapman}\thanksref{t2}\ead[label=e3]{ptrapman@math.su.se}}
\thankstext{t1}{Supported by the UK Engineering and Physical Sciences
Research Council under research Grant Number EP/E038670/1.}
\thankstext{t2}{Supported by Riksbankens Jubileumsfond of the Swedish
Central Bank and Vetenskapsr{\aa}det (Swedish Research Council) project 2010-5873.}
\runauthor{F. G. Ball, D. J. Sirl and P. Trapman}
\affiliation{University of Nottingham, University of Nottingham
and\break
Stockholm University}

\address[1]{F. Ball\\
School of Mathematical Sciences\\
University of Nottingham\\
University Park, Nottingham NG7 2RD\\
United Kingdom\\
\printead{e1}}

\address[2]{D. Sirl\\
School of Mathematics\\
Loughborough University\\
Loughborough, Leicestershire LE11 3TU\\
United Kingdom\\
\printead{e2}}

\address[3]{P. Trapman\\
Department of Mathematics\\
Stockholm University\\
SE-106 91 Stockholm\\
Sweden\\
\printead{e3}}
\end{aug}

\received{\smonth{4} \syear{2012}}
\revised{\smonth{6} \syear{2013}}

%
\begin{abstract}
In this paper we consider a model for the spread of a stochastic SIR
(Susceptible $\to$ Infectious $\to$ Recovered) epidemic on a network of
individuals described by a random intersection graph. Individuals
belong to a random number of cliques, each of random size, and
infection can be transmitted between two individuals if and only if
there is a clique they both belong to. Both the clique sizes and the
number of cliques an individual belongs to follow mixed Poisson
distributions. An infinite-type branching process approximation (with
type being given by the length of an individual's infectious period)
for the early stages of an epidemic is developed and made fully
rigorous by proving an associated limit theorem as the population size
tends to infinity. This leads to a threshold parameter $R_*$, so that
in a large population an epidemic with few initial infectives can give
rise to a large outbreak if and only if $R_* > 1$. A functional
equation for the survival probability of the approximating
infinite-type branching process is determined; if $R_* \le1$, this
equation has no nonzero solution, while if $R_*>1$, it is shown to
have precisely one nonzero solution. A law of large numbers for the
size of such a large outbreak is proved by exploiting a single-type
branching process that approximates the size of the susceptibility set
of a typical individual.
\end{abstract}

%
\begin{keyword}[class=AMS]
\kwd[Primary ]{60K35}
\kwd{92D30}
\kwd{05C80}
\kwd[; secondary ]{60J80}
\kwd{91D30}
\end{keyword}
\begin{keyword}
\kwd{Epidemic process}
\kwd{random intersection graphs}
\kwd{multi-type branching processes}
\kwd{coupling}
\end{keyword}

\end{frontmatter}

\section{Introduction}
Traditional models for the spread of SIR (Susceptible $\to$ Infectious
$\to$ Recovered) epidemics \cite{Ande00,Diek00} are based on the
homogeneous mixing assumption, that is, all pairs of individuals in the
population contact each other at the same rate, independently of each
other. Generalizations of this model have been proposed by introducing
household structure into the population \cite{Ball97}, where contacts
between household members are more frequent than other contacts; by
introducing a (social) network structure \cite{Ande99,Newm02}, where
contacts are only possible between pairs of individuals that share a
connection in the network; or both~\cite{Ball09,Ball10}. In most
models for epidemics on networks, the network is modelled by a random
graph constructed via the configuration model \cite{Moll95},
Chapter~3 of \cite{Durr06}.
In this construction one can control the degree distribution of the
vertices, but the resulting network is locally tree-like,\vadjust{\goodbreak} in the sense
that the network contains hardly any cliques (small completely
connected groups) or short loops. In real social networks, cliques are
not sparse: ``the friends of my friends are likely to be my friends as
well''. This feature of networks has been captured (among other models,
such as those in \cite{Pen03,Newm09,Glee09}) by random intersection
graphs, introduced in \cite{Karo99} 
and further studied in, for example, \cite{Brit08,Deij09,Shan10}
(see \cite{Boll10} for a related model). Random intersection graphs may
be seen as models for overlapping groups/cliques, in which a contact
between two individuals is possible only if there is a group to which
they both belong. These graphs are also known as random key graphs in
computer science \cite{Jawo09} and are related to Rasch models \cite
{Rasc61} in the social sciences.
In our paper, and in most random intersection graph models in the
literature, the resulting graph still has a tree-like structure, though
now at the level of cliques. This structure allows for analysis, but
arguably only captures some features of real (social) networks. It is
possible to make the graphs more realistic by incorporating spatial
location \cite{Gupt08}, but this makes the model intractable for our purposes.

The aim of this paper is to study SIR epidemics on random intersection
graphs. Specifically, we use branching process approximations to derive
(i) a threshold parameter $R_*$, which determines whether an epidemic
with few initial infectives can become established and infect a
nonnegligible proportion of the population, an event we call a large
outbreak, (ii) the probability that a large outbreak occurs and (iii)
the fraction of the population that is infected by a large outbreak.
These approximations are made fully rigorous as the population size
tends to infinity by proving associated limit theorems.

The only previous rigorous study of epidemics on random intersection
graphs is \cite{Brit08}. We extend the analysis of \cite{Brit08} in
three directions. First, we allow more general distributions for both
group size and the number of groups a typical individual belongs to.
In \cite{Brit08}, both of these quantities follow Poisson
distributions; here we allow them to follow mixed-Poisson
distributions. Moreover, as discussed in Section~\ref{disc}, we expect
similar results to hold when they both follow quite general
distributions, though our proofs are valid only for the mixed-Poisson case.
Secondly, we allow for an arbitrary infectious period distribution,
unlike in \cite{Brit08} where a Reed--Frost-type
model (Section~1.2 of \cite{Ande00}), which effectively has a constant infectious period, is
used. Thirdly, we give a formal proof of a law of large numbers for the
final outcome of a large outbreak, a result that was conjectured but
not proved in \cite{Brit08}. Introducing variable infectious periods
significantly complicates the analysis.
We note that for random infectious periods, our model is not covered
by Section~5 of \cite{Boll10}, since we need directed inhomogeneous
random graphs and the proofs in \cite{Boll10} rely heavily on the
structure of undirected graphs. Therefore, we need to develop
alternative techniques to determine the fraction of the population that
is infected by a large outbreak.

The remainder of the paper is organized as follows. Section~\ref{epRIG}
gives a brief introduction to random intersection graphs and SIR
epidemics defined upon them.
The main results of the paper, together with associated heuristic
explanations, are given in Section~\ref{Mainresults}. In particular, in
Section~\ref{earlystages} we show how the early stages of an epidemic
in our model can be approximated by a multitype (forward) branching
process (whose type space is in general uncountable), yielding a
threshold parameter $R_*$ and the approximate probability of a large
outbreak. In Section~\ref{finaloutcome}, a single-type (backward)
branching process, which enables the proportion of the population that
is infected by a large outbreak to be determined, is described. The key
limit theorems of the paper are stated in Section~\ref{SIRRIG}. They
show that, if there are few initial infectives, then in a large
population: (i) a large outbreak occurs with nonzero probability if
and only if the forward branching process is supercritical; (ii) the
probability that a large outbreak occurs is close to the probability
that the forward branching process survives; (iii) if there is a large
outbreak, then the proportion of the population that is infected by the
epidemic is close to the survival probability of the backward branching
process. The forward branching process is studied in Section~\ref{forwardBP}, where it is shown that the process survives with nonzero
probability if and only if $R_*>1$ and that the survival probability
may be obtained using a functional equation, which, as is proved in
Appendix \ref{appendix1}, has at most one nonzero solution. The limit
theorems corresponding to the forward and backward branching processes
are proved in Sections~\ref{pextinctionthm} and \ref{pfinalsizethm},
respectively.
Extension to more general distributions of clique size and the number
of groups a typical individual belongs to is
discussed briefly in Section~\ref{disc}. Explicit expressions, in terms
of Gontcharoff polynomials, for $R_*$ and for the
probability generating function(als) of the offspring distributions of
the backward and forward branching processes (which enable the survival
probabilities of these processes to be computed) are derived in
Appendix \ref{FSRV}.

\section{Random intersection graphs and epidemics thereon}\label{epRIG}

\subsection{Notation}
Throughout, $\mathbb{N}$ denotes the set of natural numbers not
including~$0$, while $\mathbb{Z}_+ = \mathbb{N} \cup\{0\}$.
For $x\geq0$, $\lfloor x \rfloor= \max(y \in\mathbb{Z}_+ \dvtx y\leq x)$
is the floor of $x$, and $\lceil x \rceil= \min(y \in\mathbb{Z}_+ \dvtx y\geq x)$ is the ceiling of $x$.

Furthermore, we write
\begin{eqnarray*}
f(x)&=&O\bigl(g(x)\bigr) \qquad \mbox{if }  \limsup_{x\to\infty}
\bigl|f(x)/g(x)\bigr| < \infty,
\\
f(x)&=&o\bigl(g(x)\bigr) \qquad\mbox{if }  \lim_{x\to\infty} f(x)/g(x) = 0
\qquad\mbox {and}
\\
f(x)&=&\Theta\bigl(g(x)\bigr) \qquad \mbox{if }  0 < \liminf_{x\to\infty}
\bigl|f(x)/g(x)\bigr| \leq\limsup_{x\to\infty} \bigl|f(x)/g(x)\bigr| < \infty.
\end{eqnarray*}
A (directed or undirected) graph is \textit{simple} if it contains no
parallel edges (edges that share both end-vertices) or self-loops
(edges with only one end-vertex). In a directed graph, edges are
parallel if they share both end-vertices and have the same direction.
In a \textit{multi-graph} self-loops and parallel edges are allowed. We
may construct a directed graph from an undirected one by replacing
every undirected edge by two directed edges with the same end-vertices
but having opposite directions. If we construct a simple graph from a
multi-graph, we do this by merging parallel edges and removing self-loops.

We use $\mathbb{P}$ for general unspecified probability measures, for
which the interpretation is clear from the context, and $\mathbb{E}$
for the associated expectation.
We use $\mathbb{E}_{X}$ to denote expectation with respect to the
random variable $X$. However, if no confusion is possible we sometimes
drop the subscript.
For the nonnegative random variable $X$, a mixed-Poisson($X$) random
variable, $Y$, is defined by $\mathbb{P}(Y=k) = \mathbb{E}_X[\frac
{X^k}{k!} \mathrm{e}^{-X}]$, for $k \in\mathbb{Z}_+$.
We say that a random variable is $\Poisson(x)$ if it is Poisson
distributed with mean $x$ and $\mathcal{MP}(X)$ if it has a
mixed-Poisson($X$) distribution.
We use $\tilde{X}$ to denote the size-biased variant of the
nonnegative random variable $X$, so, provided $\mathbb{E}[X] \in
(0,\infty)$, for $x\geq0$ we have
%
%
\begin{equation}
\label{sibidi} \mathbb{P}(\tilde{X} \leq x) = \frac{\int_{y \in[0,x]} y  \mathbb
{P}(X \in\mathrm{d}y)}{\mathbb{E}[X]} =
\frac{\mathbb{E}[X\ind(X \leq
x)]}{\mathbb{E}[X]}.
\end{equation}
%
Here $\ind(\mathcal{A})$, is the indicator function of $\mathcal{A}$,
which is 1 if $\mathcal{A}$ holds and 0 otherwise.
Note that if $Y \sim\mathcal{MP}(X)$, then $\tilde{Y} \sim\mathcal
{MP}(\tilde{X}) +1$; 
in this situation we use the notation $\check{Y}$ to denote a random
variable with the same distribution as $\tilde{Y}-1$, so that if $Y
\sim\mathcal{MP}(X)$, then $\check{Y} \sim\mathcal{MP}(\tilde{X})$.
This implies that $\mathbb{E}[\check{Y}]=\mathbb{E}[\tilde{X}]$.
Let $X_n \Rightarrow X$ denote convergence in distribution.
By Theorem 7.2.19 of \cite{Grim92}, we know that if $X_n \Rightarrow X$,
then $\mathbb{E}[X_n\ind(X_n \leq x)] \to\mathbb{E}[X\ind(X \leq x)]$
for all points of continuity of $\mathbb{P}(X \leq x)$.
This implies that if $\mathbb{E}[X_n] \to\mathbb{E}[X]$ and $X_n
\Rightarrow X$, then $\tilde{X}_n \Rightarrow\tilde{X}$.

We also use the notation $f_X(s) = \mathbb{E}[s^X]$ ($s\in[0,1]$) for
the probability generating function of a $\mathbb{Z}_+$-valued random
variable $X$ and $\phi_X(\theta) = \mathbb{E}[\mathrm{e}^{-\theta X}]$
($\theta\geq0$) for the moment generating function of a real-valued
random variable $X$. Note that if $Y \sim\mathcal{MP}(X)$, then
$\mathbb{E}[Y]=\mathbb{E}[X]$ and $f_Y(s) = \phi_X(1-s)$. Lastly, for
any set $A$ we denote its cardinality by $|A|$.

\subsection{Random intersection graphs}
\label{RIGs}
We consider a variant of random intersection graphs \cite
{Brit08,Deij09,Karo99} constructed via a bipartite generalization of
Norros and Reittu's Poissonian random graph model \cite{Norr06}. Random
intersection graphs may be thought of as random graphs composed of
overlapping groups/cliques of individuals/vertices. We note that the
model introduced in \cite{Karo99} is more general than (the
equal-weight variant of) the model presented in this paper. 

We construct a sequence of random intersection graphs as follows.
Consider two infinite sets of vertices $V=(v_i, i \in\mathbb{N})$ and
$V'=(v'_j, j \in\mathbb{N})$. Fix a real number $\alpha>0$.
Assign independent and identically distributed (i.i.d.) weights $(A_i,
i \in\mathbb{N})$ to the vertices in $V$, all distributed as the
nonnegative random variable $A$ and, independently, i.i.d. weights
$(B_j, j \in\mathbb{N})$ to the vertices in $V'$, all distributed as
the nonnegative random variable $B$. Assume that
%
%
\begin{equation}
\label{weightcond} \mu= \mathbb{E}[A] = \alpha\mathbb{E}[B] \in(0, \infty).
\end{equation}
%
Define $L^{(n)} = \sum_{i=1}^n A_i$ and $ L'^{(n)} = \sum_{j=1}^{\lfloor\alpha n \rfloor} B_j$, though see Remark \ref{rRemark}
below. Let $(\Omega,\mathcal{F}, \nu)$ be the corresponding probability
space, where $\Omega= (\mathbb{R}_+)^{\mathbb{N}} \times(\mathbb
{R}_+)^{\mathbb{N}}$ is the product space of nonnegative real-valued
infinite sequences $(A_i, i \in\mathbb{N})$ and $(B_j, j \in\mathbb
{N})$. The $\sigma$-field $\mathcal{F}$ is generated by the finite
dimensional cylinders on $\Omega$, and $\nu$ is the appropriate
(product) measure determined by the distributions of $A$ and $B$.
We note that, by the strong law of large numbers, both $L^{(n)}/(\mu n)
\asarrows1$ and $L'^{(n)}/(\mu n) \asarrows1$ as $n \to\infty$. Here
$\asarrows$ denotes almost sure convergence with respect to the measure
$\nu$. 

For given $\omega\in\Omega$, an auxiliary sequence of random
undirected multigraphs $(\mathbb{A}^{(n)}, n \in\mathbb{N}) = (\mathbb
{A}^{(n)}(\omega), n \in\mathbb{N})$ is constructed as follows. For
each $n$, the vertex set of $\mathbb{A}^{(n)}$ consists of $V^{(n)}=
(v_i, 1 \leq i \leq n)$ and $V'^{(n)} = (v'_j, 1 \leq j \leq\lfloor
\alpha n \rfloor)$.
Vertices $v_i \in V^{(n)}$ and $v'_j \in V'^{(n)}$ share a $\Poisson
(A_i B_j/(\mu n))$ number of edges (see Remark \ref{simpleA'}).
Conditioned on the weights of vertices, that is on $\omega$, the
numbers of edges between distinct pairs of vertices are independent,
and there is no edge in $\mathbb{A}^{(n)}$ connecting vertices either
both in $V^{(n)}$ or both in $V'^{(n)}$. Note that in $\mathbb{A}^{(n)}$,
the degree of vertex $v_i \in V^{(n)}$ is $\Poisson(A^{(n)}_i)$ with
%
%
\begin{equation}
\label{intermedA} A^{(n)}_i = A_i
L'^{(n)}/(\mu n) \asarrows A_i \qquad\mbox{as $n \to
\infty$,}
\end{equation}
while the degree of vertex $v'_j \in V'^{(n)}$ is $\Poisson(B^{(n)}_j)$ with
%
%
\begin{equation}
\label{intermedB} B_j^{(n)}=B_j
L^{(n)}/(\mu n) \asarrows B_j \qquad\mbox{as $n \to\infty$.}
\end{equation}
The random variables $A^{(n)}$ and $B^{(n)}$ are defined by
%
%
\begin{eqnarray}
\mathbb{P}\bigl(A^{(n)}\leq x\bigr) & =& n^{-1}\bigl|\bigl\{1 \leq
i \leq n \dvtx A^{(n)}_i \leq x\bigr\}\bigr| \qquad(x\geq0) \quad\mbox{and}
\label{AnDef}
\\
\mathbb{P}\bigl(B^{(n)}\leq x\bigr) & =& \lfloor\alpha n
\rfloor^{-1}\bigl|\bigl\{1 \leq j \leq\lfloor\alpha n \rfloor\dvtx
B^{(n)}_j \leq x\bigr\}\bigr|\qquad(x\geq0). \label{BnDef}
\end{eqnarray}
Thus $A^{(n)}(\omega)$ and $B^{(n)}(\omega)$ are random variables with
the empirical distribution
of the rescaled weights $\{A_i^{(n)}\}$ and $\{B_j^{(n)}\}$,
respectively. By the strong law of large numbers, $A^{(n)} \Rightarrow
A$ and $B^{(n)} \Rightarrow B$ as $n \to\infty$.



For the purpose of this paper it is not important how the graphs in the
sequence depend on each other. For simplicity we assume that,
conditioned on $\omega= (A_i, i \in\mathbb{N})\times(B_j, j \in
\mathbb{N})$, the graphs $(\mathbb{A}^{(n)}, n \in\mathbb{N})$ are independent.

\begin{figure}

\includegraphics{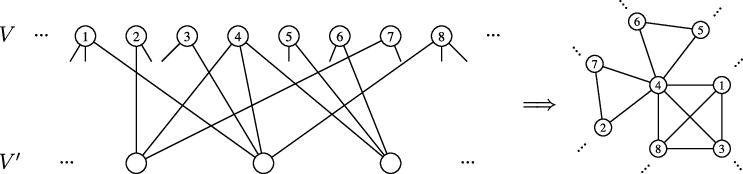}

\caption{Construction of $G^{(n)}$ from $\mathbb{A}^{(n)}$.}
\label{figrain}
\end{figure}

The vertices of the random intersection graph $G^{(n)}$ are precisely
those in $V^{(n)}$. Two (distinct) vertices share an edge in $G^{(n)}$
if and only if there is at least one path of length 2 between them in
$\mathbb{A}^{(n)}$. Thus, $G^{(n)}$ is a simple graph. This
construction is visualised in Figure~\ref{figrain}.
We note that $G^{(n)}$ is slightly different from an ordinary random
intersection graph. In \cite{Brit08,Deij09} the conditional probability
that vertices with weights $A_i$ and $B_j$ share an edge in $\mathbb
{A}^{(n)}$ is given by $\min(1,A_iB_j/(\mu n))$, as opposed to $1-\exp
[-A_iB_j/(\mu n)]$ in this paper. (Note also that in \cite{Brit08} the
weights are constant.)

\begin{remark}
\label{simpleA'}
Of course it is possible to construct a simple version of the (multi)
graph $\mathbb{A}^{(n)}$ directly, in which the vertices $v_i$ and
$v'_j$ share an edge with probability $1-\exp[-A_iB_j/(\mu n)]$.
Indeed, this is sufficient to describe the population structure of our
model. We use the present construction, where $v_i$ and $v_j$ share a
Poisson distributed number of edges, in order to have the machinery
ready for branching process approximations.
\end{remark}

\begin{remark}
The graph $G^{(n)}$ is a graph of overlapping cliques, in which,
asymptotically as $n\to\infty$, the number of cliques a vertex is part
of has an $\mathcal{MP}(A)$ distribution, and the clique sizes have an
$\mathcal{MP}(B)$ distribution. Both of these distributions have finite
mean by assumption.
\end{remark}

\begin{remark} 
\label{rRemark}
Since the random intersection graph does not change if, for some $r \in
(0,\infty)$, the random variables $A$ and $B$ are replaced by $rA$ and
$B/r$, condition (\ref{weightcond}) might be replaced by $\mathbb
{E}[A]<\infty$ and $\mathbb{E}[B] < \infty$, but this does not gain any
generality. The linear scaling $|V'^{(n)}| = \lfloor\alpha|V^{(n)}|
\rfloor$ is assumed in order to guarantee that, as $n\to\infty$, (i)
clique sizes do not grow to infinity, and (ii) two (or more) cliques
contain at most one common vertex, with high probability.\vadjust{\goodbreak}
\end{remark}

\begin{remark}
\label{remark2.2}
In this paper we make use of the following equivalent way of
constructing $\mathbb{A}^{(n)}$.
Initially all vertices are \emph{unexplored}. Pick a vertex from
$V^{(n)}$ according to some law (e.g., uniformly at random), say
vertex $v_i$, which has weight~$A_i$; this vertex becomes \emph
{active}. Assign a $\Poisson(A^{(n)}_i)$ number of edges to it
[see \eqref{intermedA}]. The end-vertices in $V'^{(n)}$ of these edges
are chosen independently with replacement and the probability that
$v'_j$ is chosen is $B_j/L'^{(n)}$. After this vertex $v_i$ is made
\emph{explored}, while the chosen vertices become active.

Now, if there are any, explore the active vertices from $V'^{(n)}$ one
by one. Suppose that we explore vertex $v'_j$, which has weight $B_j$;
then assign a $\Poisson(B_j^{(n)})$ number of edges to it. [We observe
that the number of edges between $v_j'$ and a previously unexplored
vertex $v_l$ is indeed $\Poisson(A_l B_j/(\mu n))$, independent of the
numbers of edges between all other pairs of vertices, as desired.]
These edges connect to vertices chosen independently, with replacement,
from $V^{(n)}$; vertex $v_l$ being chosen with probability
$A_l/L^{(n)}$. If the end vertex has already been explored, then the
edge is ignored and not added to the graph; otherwise it is added and
the end vertex in $V^{(n)}$ becomes active. If all the edges from
$v'_j$ are drawn, then $v'_j$ is made explored.

The next step is to pick one of the active vertices from $V^{(n)}$, if
there are any, according to some, for now unspecified, law and explore
it. Say that we choose $v_k$, which has weight $A_k$. Then we proceed
as in the first step. We assign a $\Poisson(A^{(n)}_k)$ number of edges
to it. Then the end-vertices in $V'^{(n)}$ of these edges are chosen
independently with replacement, and the probability that $v'_j$ is
chosen is $B_j/L'^{(n)}$. If the end vertex has been explored before,
then the edge is ignored and deleted. After this, vertex $v_k$ is made
explored and the newly chosen vertices in $V'^{(n)}$ which are
unexplored become active. We now explore all active vertices in
$V'^{(n)}$ in turn, and so on until there is no active vertex left.
After that an unexplored vertex from $V^{(n)}$ is chosen, and the
process goes on until all vertices in $V^{(n)}$ are explored. Note that
if after this construction 
there are unexplored vertices left in $V'^{(n)}$, they will have degree
0, since there is no end-vertex left in $V^{(n)}$ to connect to.
\end{remark}

\subsection{SIR epidemics}
\label{SIRepidemics}

We consider a stochastic SIR epidemic on the random intersection graph
$G^{(n)}$.
The vertices of the graph correspond to individuals and the edges to
relationships/possible contacts.
We assume that initially there is one infectious individual/vertex,
chosen uniformly at random from the population, while all other
individuals are susceptible. Every individual, independently of other
individuals, makes (directed) contact with each of its neighbours in
$G^{(n)}$ at the points of independent Poisson processes of unit
intensity. If an infectious individual contacts a susceptible one, the
susceptible becomes infectious. Infectious individuals stay infectious
for a random infectious period, distributed as a random variable~$\mathcal{I}$,
after which the infectious individual recovers and plays
no further part in the epidemic. Infectious periods are i.i.d. and
independent of the Poisson processes generating the contacts. An
infectious contact is a contact by an infectious individual,
irrespective of the state of the contacted individual. Note that there
is no loss of generality in assuming that the intensity of the Poisson
processes governing the contacts is 1, since this can always be
achieved by rescaling time. We denote the above epidemic model by
$\mathcal{E}^{(n)}(A,B,\mathcal{I})$.

For ease of exposition, primarily to avoid multitype branching
processes that are reducible, we assume that $\mathbb{P}(\mathcal
{I}=0)=0$. We omit the details, but our results are readily extended to
the case $\mathbb{P}(\mathcal{I}=0)>0$. Note, however, that we do allow
for the possibility that $\mathbb{P}(\mathcal{I}=\infty)>0$; if an
infectious individual has infinite infectious period, then that individual
almost surely makes infectious contact with every member of
each clique it belongs to.

In order to study properties of the epidemic on a graph, $G$ say, we
introduce the \textit{epidemic generated graph}, which is a directed
graph constructed as follows. If $G$ is undirected, then make it
directed by replacing every edge by two edges connecting the same
vertices but in opposite directions. Assign every vertex $i$ in $G$ an
independent realisation, $x_i$, of the random variable $\mathcal{I}$.
Now thin $G$ by deleting, independently, each (directed) edge emanating
from vertex $i$ with probability $\mathrm{e}^{-x_i}$. Thus an edge
starting at $v_i$ is deleted if infection would not pass along it were
$v_i$ to become infected during the epidemic. The set of vertices that
can be reached in the epidemic generated graph from an initially
infectious vertex $v_0$ (including $v_0$ itself) is distributed as the
set of ultimately recovered individuals. The set of vertices from which
there is a path in the epidemic generated graph to vertex $v_0$,
including $v_0$ itself, is said to be the \emph{susceptibility set} of
$v_0$ \cite{Ball00,Ball02}. If one of the vertices in the
susceptibility set of $v_0$ is the initially infectious individual,
then $v_0$ will be ultimately recovered in the epidemic.

\section{Main results and heuristics}
\label{Mainresults}

\subsection{Introduction}
\label{mainintro}

In this section we outline the main results of the paper, together with
their heuristic explanations. In Section~\ref{earlystages}, we explain
how the early stages of an SIR epidemic on a random intersection graph
may be approximated by a (forward) branching process, yielding a
threshold parameter $R_*$ [see \eqref{R*def}] for the epidemic and the
approximate probability that such an epidemic becomes established when
the population size $n$ is large. Unless the infectious period $\mathcal
{I}$ is constant, this branching process is multitype, its type space
being the support of $\mathcal{I}$ and hence in general uncountable.
This infinite type branching process is studied separately in
Section~\ref{forwardBP}. In Section~\ref{finaloutcome}, we show how the
susceptibility set of an individual may be approximated by a (backward)
branching process, which is single-type even if $\mathcal{I}$ is not constant.
Furthermore, we explain why, if $n$ is large, the proportion of the
population that is infected during an epidemic that becomes
established is approximately the probability that the backward
branching process avoids extinction. The above approximations are made
fully rigorous by considering SIR epidemics on a sequence of random
intersection graphs, indexed by the population size $n$, and proving
associated limit theorems. These theorems are stated in Section~\ref{SIRRIG} and proved in Section~\ref{proofs}. Calculation of extinction
probabilities for the forward and backward branching processes requires
exact results concerning the final outcome and susceptibility sets for
standard SIR epidemics in closed homogeneously mixing populations,
which are given in Appendix \ref{FSRV}.

\subsection{Early stages of an epidemic}
\label{earlystages}

\subsubsection{Fixed infectious period}
\label{earlyStagesFixedIP}
Consider the epidemic model $\mathcal{E}^{(n)}(A,B,\mathcal{I})$
defined in Section~\ref{SIRepidemics} and, for simplicity, suppose
first that the infectious period is constant, that is, there exists
$\iota>0$ such that $\mathbb{P}(\mathcal{I}=\iota)=1$.
In the limit as the population size $n \to\infty$, the initial
infective, $i^*$ say, belongs to $X\sim\mathcal{MP}(A)$ cliques, having
sizes $\check{Y}_1+1,\check{Y}_2+1,\ldots,\check{Y}_X+1$, where, given
$X$, the random variables $\check{Y}_1,\check{Y}_2,\ldots,\check{Y}_X$
are mutually independent
and $(\check{Y}_i|X)\sim\mathcal{MP}(\tilde{B})$ $(i=1,2,\ldots,X)$.
The size biasing comes in because the probability of being part of a
clique is proportional to its weight. 
Moreover, apart from $i^*$, these cliques are almost surely disjoint as
$n \to\infty$. The initial infective will trigger a local
(within-clique) epidemic in each of the $X$ cliques it belongs to. The
group of initial susceptibles in a single clique that are infected
through a local epidemic started by $i^*$ is called a \emph{litter} of
$i^*$. (Note that a litter may be empty; this happens if no susceptible in the
corresponding clique is infected.) Let $T(m)$ denote the size of a
typical litter, not counting the initial infective $i^*$, given that
the clique has size $m+1$. [We call $T(m)$ the size of a local epidemic
or the size of a litter.] Then the total number of individuals infected
(excluding $i^*$) by the local epidemics in the cliques that $i^*$
belongs to is distributed as
\[
C^f=\sum_{i=1}^X T(
\check{Y}_i),
\]
where $T(\check{Y}_1),T(\check{Y}_2),\ldots,T(\check{Y}_X)$ are
independent, since the infectious period is constant.

Now consider a typical individual, $j^*$ say, that is part of one of
the litters of~$i^*$. 
In the limit as $n \to\infty$, (i) individual $j^*$ belongs to $\check
{X}\sim\mathcal{MP}(\tilde{A})$ cliques, in addition to the clique
$j^*$ was infected through (i.e., the one also containing~$i^*$),
having sizes distributed independently as $\mathcal{MP}(\tilde{B})+1$
and (ii) apart from $j^*$, the $\check{X}+1$ cliques containing $j^*$
are disjoint. (The size biasing here
arises because, in the construction of $G^{(n)}$, the probability that
a vertex joins a given clique is proportional to the weight of that
vertex; see Remark \ref{remark2.2}.) Individual $j^*$ will trigger a
local epidemic in each of the $\check{X}$ ``new'' cliques it belongs to.
The total number of individuals infected (excluding $j^*$) in these
$\check{X}$ local epidemics (the sum of the sizes of the litters of
$j^*$) is distributed as
\[
\tilde{C}^f=\sum_{i=1}^{\check{X}} T(
\check{Y}_i),
\]
where, given $\check{X}$, the random variables $T(\check{Y}_1),T(\check
{Y}_2),\ldots,T(\check{Y}_{\check{X}})$ are independent.

The construction of the epidemic process may be continued in the
obvious fashion. It follows that, if the population size $n$ is large,
the number of infected individuals in the early stages of the epidemic
process may be approximated by a (Galton--Watson) branching process,
with one ancestor, and offspring distribution distributed as
$C^f$ in the initial generation and as $\tilde{C}^f$ in all subsequent
generations. This approximation is made precise by using a coupling
argument in Section~\ref{pextinctionthm}. The coupling between the
epidemic and branching processes breaks down when a clique used to
spread a local epidemic intersects a previously used clique, which,
with probability tending to one as $n\to\infty$, happens if and only if
the branching process does not go extinct.

Let
%
%
\begin{equation}
\label{R*def} R_*=\mathbb{E}\bigl[\tilde{C}^f\bigr]=
\mathbb{E}_{\check{Y}}\bigl[\mathbb{E}\bigl[T(\check {Y})\mid\check{Y}\bigr]
\bigr]\mathbb{E}[\check{X}] =\mathbb{E}_{\check{Y}}\bigl[\mathbb{E}\bigl[T(
\check{Y})\mid\check{Y}\bigr]\bigr]\mathbb {E}[\tilde{A}],
\end{equation}
and, for $s \in[0,1]$, let
\[
f_{C^f}(s)=\mathbb{E}\bigl[s^{C^f}\bigr]=f_X\bigl(
\mathbb{E}_{\check{Y}}\bigl[f_{T(\check
{Y})|\check{Y}}(s)\bigr]\bigr)
\]
and
\[
f_{\tilde{C}^f}(s)=\mathbb{E}\bigl[s^{\tilde{C}^f}\bigr]=f_{\check{X}}
\bigl(\mathbb {E}_{\check{Y}}\bigl[f_{T(\check{Y})|\check{Y}}(s)\bigr]\bigr).
\]
Let $\rho$ be the survival probability of the above branching process
(i.e., the probability that it does not go extinct).
Then, by standard branching process
theory (Theorem 2.3.1 of \cite{Jage75}), if $R_*\le1$, then $\rho=0$, and if $R_*>1$, then
%
%
\begin{equation}
\label{rhoconst} \rho=1-f_{C^f}(\sigma),
\end{equation}
where $\sigma$ is the unique solution in $[0,1)$ of the equation
%
%
\begin{equation}
\label{extinctionprobconst} f_{\tilde{C}^f}(s)=s.
\end{equation}
The coupling of the epidemic and branching processes mentioned above
implies that, if the population size $n$ is suitably large, $R_*$ is a
threshold parameter for the epidemic process and
the probability that an epidemic initiated by a single infective
becomes established and leads to a major outbreak is given
approximately by~$\rho$. Note that in \cite{Brit08}, the notation $R_0$
is used instead of $R_*$. We use the notation of~\cite{Ball09,Ball10}
because $R_0$ is usually defined as the expected number of new \emph
{direct} infections caused by an infectious individual in the first
stages of an epidemic \cite{Ande00,Diek00,Pell12},
while in \eqref{R*def} \textit{all} individuals infected by a local
epidemic are ``assigned to'' the initial infectious individual in the clique.

\subsubsection{General infectious period distribution}
\label{earlyStagesGeneralIP}

When the infectious period is not constant we can still approximate the
epidemic $\mathcal{E}^{(n)}(A,B,\mathcal{I})$ by considering successive
local epidemics as above, but the approximating process is no longer a
simple single-type branching process. There are two reasons for this.
First, the sizes of the litters of an individual, $i^*$ say,
are not independent since the infectious period of the initial
infective in the corresponding cliques is the same (i.e., the
infectious period of $i^*$). Secondly, the infectious periods of
infectives in a litter are not independent of the size of that litter.
These difficulties may be overcome by considering a multitype branching
process, in which individuals are typed by the length of their
infectious period. If the infectious period $\mathcal{I}$ has finite
support, then standard finite-type branching process theory (see,
e.g., Chapter~4 of \cite{Jage75}) may be used, so we now assume that
$\mathcal{I}$ has infinite (possibly uncountable) support. For ease of
exposition we assume that $\mathcal{I}$ has support $(0,\infty]$.

In view of these observations, we approximate the early stages of the
epidemic $\mathcal{E}^{(n)}(A,B,\mathcal{I})$ by a multitype branching process
\[
\mathcal{Z}^f = \mathcal{Z}^f(A,B,\mathcal{I}) = \bigl(
\mathcal{Z}^f_i, i \in\mathbb{Z}_+\bigr),\vadjust{\goodbreak}
\]
defined as follows. The type space is $(0,\infty]$, with the type of an
individual being given by the infectious period of the corresponding
individual in the epidemic process. For $i \in\mathbb{Z}_+$, $\mathcal
{Z}^f_i$ is a multiset of points in $(0,\infty]$ giving the types of
individuals present in generation $i$ of the branching process. (Note
that if the distribution of $\mathcal{I}$ has atoms, at infinity or
otherwise, then $\mathcal{Z}^f_i$ may contain repeated elements; on the
other hand if the distribution of $\mathcal{I}$ is continuous, then,
almost surely, all elements of $\mathcal{Z}^f_i$ are distinct, and
hence $\mathcal{Z}^f_i$ is a set.)
There is one ancestor, corresponding to the initial infective,
$i^*$ say, in the epidemic $\mathcal{E}^{(n)}(A,B,\mathcal{I})$, and
its type is distributed as $\mathcal{I}$. As in the constant infectious
period case, $i^*$ belongs to $X\sim\mathcal{MP}(A)$ cliques, having sizes
distributed independently as $\check{Y}+1$, where $\check{Y}\sim\mathcal
{MP}(\tilde{B})$, and the offspring of the
ancestor in $\mathcal{Z}^f$ corresponds 
to all the individuals infected in the local epidemics triggered by
$i^*$ in these $X$ cliques, though now of course we also keep track of
their types (infectious periods). In the branching process, a group of
children corresponding to a litter in the epidemic process is also
referred to as a litter. 
The offspring of any individuals in a noninitial generation of
$\mathcal{Z}^f$ are defined in a similar fashion, except $X$ is
replaced by $\check{X}\sim\mathcal{MP}(\tilde{A})$. Of course, the
offspring of distinct individuals in $\mathcal{Z}^f$ are mutually independent.

The branching process $\mathcal{Z}^f$, which we call a forward
branching process because it approximates the forward spread of the
epidemic $\mathcal{E}^{(n)}(A,B,\mathcal{I})$, is analysed in
Section~\ref{forwardBP}.
Let $\tilde{\mathcal{Z}}^f$ be the multitype branching process defined
analogously to $\mathcal{Z}^f$, except the offspring distribution in
all generations of $\tilde{\mathcal{Z}}^f$ is that of the noninitial
generations in $\mathcal{Z}^f$. Let $\rho$ be the probability that
$\mathcal{Z}^f$ survives and, for $x\in(0,\infty]$, let $\tilde{\rho
}(x)$ be the probability that $\tilde{\mathcal{Z}}^f$ survives given
that the ancestor has type $x$. Let $R_*$ be defined as in \eqref
{R*def}, where $T(m)$ is distributed as 
the size of a local epidemic, initiated by a single infective in a
clique of size $m+1$, in which the infectious periods of infectives
(including the initial one) are i.i.d. copies of $\mathcal{I}$. (An
expression for $\mathbb{E}_{\check{Y}}[\mathbb{E}[T(\check{Y})\mid\check
{Y}]]$ is given by equation \eqref{elocal} in Appendix \ref{application},
thus enabling $R_*$ to be computed.) Then $\rho>0$ if
and only if $R_*>1$ (see Theorem \ref{survandR0}), so $R_*$ is still a
threshold parameter for the epidemic. Also, when $R_*>1$, $\rho$ is
given by an infinite-type analogue of \eqref{rhoconst}; see \eqref
{survive}, which expresses $\rho$ as the expectation of a functional of
$\tilde{\rho}$ with respect to the distribution $\mathcal{I}$ of $x$.
Furthermore, $\tilde{\rho}$ satisfies a functional equation [see \eqref
{survivetilde}], which is essentially an infinite-type analogue
of~\eqref{extinctionprobconst} and has at most one nonzero solution (see
Lemma \ref{unilem}).

\subsection{Final outcome of an epidemic}
\label{finaloutcome}

Recall the definition of the susceptibility set of an individual given
in Section~\ref{SIRepidemics}. We require also the concept of a \emph
{local susceptibility set}, 
which is defined in exactly the same way as a susceptibility set but
for an epidemic on a single clique. For $m=0,1,\ldots,$ let $S(m)$
denote the size of a typical local susceptibility set of an individual
amongst the $m$ other individuals
in a clique of size $m+1$.

We may approximate the early growth of a susceptibility set of an
individual, $i^*$ say,\vadjust{\goodbreak} by a branching process in much the same way as
for the early stages of an epidemic. We consider first those
individuals, not including $i^*$ itself, who belong to a local
susceptibility set of $i^*$. These are the offspring of $i^*$ in the
branching process. We next repeat this process for each individual,
$j^*$ say, in the first generation of the branching process to
obtain the second generation, and so on. (When determining the
offspring of $j^*$, we need only consider its local susceptibility set
in cliques other than that which contains $i^*$; any individual
in $j^*$'s local susceptibility set who is in that clique has already
been counted as part of $i^*$'s local susceptibility set.) In the limit
as $n\to\infty$, this leads to a (backward) branching process
\[
\mathcal{Z}^b = \mathcal{Z}^b(A,B,\mathcal{I}) = \bigl(
\mathcal{Z}^b_i, i \in\mathbb{Z}_+\bigr)
\]
having one ancestor, in which the number of offspring of the
ancestor is distributed as
\[
C^b=\sum_{i=1}^X S(
\check{Y}_i),
\]
and the number of offspring of any subsequent individual is distributed as
\[
\tilde{C}^b=\sum_{i=1}^{\check{X}} S(
\check{Y}_i),
\]
where $X, \check{X}, \check{Y}_1,\check{Y}_2,\ldots $ are independent,
$X\sim\mathcal{MP}(A)$, $\check{X}\sim\mathcal{MP}(\tilde{A})$ and
$\check{Y}_i\sim\mathcal{MP}(\tilde{B})$ $(i=1,2,\ldots)$.

Note that the local susceptibility set of an individual is independent
of its infectious period, so $\mathcal{Z}^b$ is a single-type branching
process; in contrast to $\mathcal{Z}^f$, which is
single-type only if $\mathcal{I}$ is almost surely equal to a fixed constant.

Let
%
%
\begin{equation}
\label{R*bdef} R_*^b=\mathbb{E}\bigl[\tilde{C}^b\bigr]=
\mathbb{E}_{\check{Y}}\bigl[\mathbb{E}\bigl[S(\check {Y})\mid\check{Y}\bigr]
\bigr]\mathbb{E}[\tilde{A}]
\end{equation}
be the mean number of children of an individual in $\mathcal{Z}^b$ who
is not the ancestor
and, for $s \in[0,1]$, define the probability generating functions
\[
f_{C^b}(s)=\mathbb{E}\bigl[s^{C^b}\bigr]=f_X\bigl(
\mathbb{E}_{\check{Y}}\bigl[f_{S(\check
{Y})|\check{Y}}(s)\bigr]\bigr)
\]
and
\[
f_{\tilde{C}^b}(s)=\mathbb{E}\bigl[s^{\tilde{C}^b}\bigr]=f_{\check{X}}
\bigl(\mathbb {E}_{\check{Y}}\bigl[f_{S(\check{Y})|\check{Y}}(s)\bigr]\bigr).
\]
Denote by $\rho^b=\rho^b(A,B,\mathcal{I})$ the survival probability of
$\mathcal{Z}^b$. Then, by standard branching process theory, if $R_*^b
\le1$, then $\rho^b=0$, and if $R_*^b > 1$, then
%
%
\begin{equation}
\label{rhob} \rho^b=1-f_{C^b}(\xi),\vadjust{\goodbreak}
\end{equation}
where $\xi$ is the unique solution in $[0,1)$ of the equation
\[
f_{\tilde{C}^b}(s)=s.
\]
Note that an expression for $\mathbb{E}_{\check{Y}}[f_{S(\check
{Y})|\check{Y}}(s)]$ is given by equation \eqref{suslocal} in
Appendix~\ref{application}, which enables $\rho^b$ to be computed. In
connection with this computation, also recall that $f_X(s) = \phi
_A(1-s)$ and observe that $f_{\check{X}}(s)=\phi_{\tilde{A}}(1-s)=-\phi
'_A(1-s)/\mathbb{E}[A]$, where $\phi'_A$ is the derivative of $\phi_A$.

Before describing how the backward branching process $\mathcal{Z}^b$ is
used to study the final outcome of an epidemic in a large population,
we discuss briefly the relationship between the forward and backward
branching processes. In particular we note two important consequences
of this relationship.

\begin{remark}
Let $G'$ be the epidemic generated graph (see Section~\ref{SIRepidemics}) for an epidemic on a single clique ($G$ say) of $m+1$
individuals, labelled $0,1,\ldots,m$. For distinct $i,j \in\{0,1,\ldots
,m\}$, let $\chi_{i,j}=1$ if there is a chain of directed edges from
$i$ to $j$ in $G'$ and let $\chi_{i,j}=0$ otherwise. Then $T(m)$ and
$S(m)$ are distributed as $\sum_{i=1}^m \chi_{0,i}$ and $\sum_{i=1}^m
\chi_{i,0}$, respectively, so by symmetry, $\mathbb{E}[T(m)]=m\mathbb
{P}(\chi_{0,1}=1)$ and $\mathbb{E}[S(m)]=m\mathbb{P}(\chi_{1,0}=1)$.
Further, by symmetry, $\mathbb{P}(\chi_{0,1}=1)=\mathbb{P}(\chi
_{1,0}=1)$, and it follows from \eqref{R*def} and \eqref{R*bdef} that
$R_*^b=R_*$. Thus we use only the notation~$R_*$.
\end{remark}

\begin{remark}
Consider the graphs $G$ and $G'$ of the previous remark, and suppose
that the infectious period $\mathcal{I}$ is constant, say $\mathbb
{P}(\mathcal{I}=\iota)=1$. Then $G'$ is obtained from the directed
version of $G$ by deleting directed edges independently, each with
probability $\mathrm{e}^{-\iota}$. Thus, if $G''$ is obtained from $G'$ by
reversing the direction of all arrows, then $G''$ and $G'$ are
identically distributed, whence so are $T(m)$ and $S(m)$. It follows
that in this case $\rho^b=\rho$. This argument breaks down when
$\mathcal{I}$ is not constant. In that case, apart from the branching
process $\mathcal{Z}^f$ being multitype, the presence/absence of
directed edges from a given vertex in $G'$ are not independent, whence
$T(m)$ and $S(m)$ have different distributions. Thus generally $\rho^b
\ne\rho$.
\end{remark}

Now we describe informally the relationship between the backward
branching process and the final outcome of an epidemic. (This
description assumes that there is no vertex with weight greater than
$\log n$; the full argument is given in Section~\ref{pfinalsizethm}.)
Consider the epidemic model $\mathcal{E}^{(n)}(A,B,\mathcal{I})$, and
suppose that the population size $n$ is large. Choose an initially
susceptible individual uniformly at random from all initial
susceptibles, $j$ say, and construct its susceptibility set on a
generation basis as described above for $\mathcal{Z}^b$. Stop this
construction after $t_n=\lceil\log\log n\rceil$ generations or when the
susceptibility set process goes extinct, whichever occurs first. The
susceptibility set process can be coupled to the backward branching
process $\mathcal{Z}^b$ so that, with probability tending to $1$ as $n
\to\infty$, the two coincide over generations $0,1,\ldots,t_n$ and
their common size at generation $t_n$ is not greater than $n^\varepsilon$
for any $\varepsilon>0$ (cf. the start of the proof of Lemma \ref{sizesus}). Also, if $R_*>1$, there exists $c>0$ such that the
probability that $Z^b_{t_n}>(\log n)^c$ tends to $\rho^b$ as $n \to
\infty$ (cf. Lemma \ref{sizesus}).

By symmetry, the initial infective in $\mathcal{E}^{(n)}(A,B,\mathcal
{I})$, $i$ say, may be chosen by picking an individual uniformly at
random from the population excluding $j$. Thus, if $j$'s susceptibility
set process goes extinct before reaching generation $t_n$ then the
probability that $j$'s susceptibility set contains the initial
infective (and hence that $j$ is infected during the epidemic)
tends to zero as $n \to\infty$. Suppose instead that $j$'s
susceptibility set process does reach generation $t_n$. Then we choose
the initial infective $i$ as above, construct the forward epidemic
process from $i$ and determine whether or not the latter intersects the
[at least $(\log n)^c$ and at most $n^\varepsilon$] individuals in
generation $t_n$ of $j$'s partially constructed susceptibility set. If
it does then $j$ is infected during the epidemic, otherwise $j$
remains uninfected.

Recall that the forward epidemic process originating from $i$ is
approximated by the branching process $\mathcal{Z}^f$. If $\mathcal
{Z}^f$ goes extinct, then in the limit as $n \to\infty$, there are
only finitely many individuals infected in the epidemic and hence the
probability that the epidemic intersects generation $t_n$ of $j$'s
partially constructed susceptibility set tends to zero. If $\mathcal
{Z}^f$ does not go extinct, then by exploiting a lower bounding
branching process for the epidemic process, we show in Section~\ref{pfinalsizethm} that, as $n \to\infty$, the epidemic process almost
surely infects $\Theta(n)$ individuals and hence the probability that
it intersects generation $t_n$ of $j$'s partially constructed
susceptibility set tends to one.

The above implies that the asymptotic probability that an initial
susceptible, chosen uniformly at random, is infected during a
major outbreak is $\rho^b$. Hence the asymptotic expected proportion of
the population infected during a major outbreak is also $\rho
^b$. Now consider two distinct initial susceptibles chosen uniformly at
random, $j_1$ and $j_2$ say, and construct their susceptibility sets on
a generation basis as above, stopping each process after $t_n$
generations or if it goes extinct. 
The two partially constructed susceptibility set processes are
asymptotically independent as $n \to\infty$, which enables a weak law
of large numbers to be proved for the proportion of the population that
is infected during a major outbreak.

\subsection{Limit theorems for SIR epidemics on random intersection
graphs}\label{SIRRIG}

Let $\mathcal{R}^{(n)}=\mathcal{R}^{(n)}(A,B,\mathcal{I})$ be the set
of ultimately recovered vertices, including the single initial
infective, in the SIR epidemic $\mathcal{E}^{(n)}(A,B,\mathcal{I})$ on
the random intersection graph $G^{(n)}$, constructed using the
infectious period distribution $\mathcal{I}$ and the sequences $(A_i, i
\in\mathbb{N})$, $(B_j, j \in\mathbb{N})$ (as described in
Section~\ref{RIGs}). Our focus is on the properties of $|\mathcal
{R}^{(n)}|$, the number of ultimately recovered individuals in the
epidemic. For a branching process, $\mathcal{Z}^f$ say, let $|\mathcal
{Z}^f| = \sum_{i=0}^\infty|\mathcal{Z}^f_i|$ denote its total size
(total progeny), including the ancestor.
Recall that $\mathcal{Z}^f=\mathcal{Z}^f(A,B,\mathcal{I})$ and $\mathcal
{Z}^b=\mathcal{Z}^b(A,B,\mathcal{I})$ are the (forward and backward)
branching processes, which approximate the epidemic process and the
process exploring a susceptibility set, respectively. Recall also that
$\rho$ and $\rho^b$ are their respective survival probabilities. 

Our first theorem establishes the precise sense in which the forward
process approximates the early stages of an epidemic.

\begin{theorem}\label{extinctionthmu}
For all $k \in\mathbb{N}$,
\[
\lim_{n\to\infty} \mathbb{P}\bigl(\bigl|\mathcal{R}^{(n)}\bigr|=k
\bigr) = \mathbb {P}\bigl(\bigl|\mathcal{Z}^f\bigr|=k\bigr).
\]
\end{theorem}

The next result establishes the connection between the backward process
and the proportion of individuals ultimately recovered.
%
\begin{theorem}\label{finalsizethmu}
Suppose that $R_*>1$. Then for every $0 < \varepsilon< \rho^b$,
\[
\lim_{n\to\infty} \mathbb{P} \biggl( \biggl\llvert \frac{|\mathcal{R}^{(n)}|}{n}
- \rho^b\biggr\rrvert < \varepsilon \biggr) = \rho.
\]
\end{theorem}

Theorems \ref{extinctionthmu} and \ref{finalsizethmu} are proved in
Sections~\ref{pextinctionthm} and \ref{pfinalsizethm}, respectively.
Finally, we use these two results to establish the following
convergence in distribution of the proportion of individuals ultimately
recovered in the epidemic process.

\begin{theorem}\label{finalsizecdthm}
Let $T_F$ be a random variable with $\mathbb{P}(T_F=\rho^b)=\rho
=1-\mathbb{P}(T_F=0)$. Then, as $n \to\infty$,
\[
n^{-1}\bigl|\mathcal{R}^{(n)}\bigr| \Rightarrow T_F.
\]
\end{theorem}

\begin{pf}
First note that Theorem \ref{extinctionthmu} implies that, for any
$\varepsilon> 0$ and any $k \in\mathbb{N}$,
\[
\liminf_{n \to\infty}\mathbb{P} \bigl(n^{-1} \bigl|
\mathcal{R}^{(n)}\bigr| \le \varepsilon \bigr) \ge\mathbb{P}\bigl(\bigl|
\mathcal{Z}^f\bigr| \le k\bigr),
\]
whence, letting $k \to\infty$,
%
%
\begin{equation}
\label{limpleneps1} \liminf_{n \to\infty}\mathbb{P} \bigl(n^{-1}
\bigl|\mathcal{R}^{(n)}\bigr| \le \varepsilon \bigr) \ge1-\rho.
\end{equation}

Suppose that $R_* \le1$. Then $\rho=0$ and \eqref{limpleneps1} implies that
%
%
\begin{equation}
\label{subfinalsizecd1} n^{-1}\bigl|\mathcal{R}^{(n)}\bigr| \Rightarrow0\qquad \mbox{as
$n \to\infty$}.
\end{equation}
On the other hand, suppose that $R_* > 1$, so $\rho>0$ and $\rho^b>0$.
Then Theorem~\ref{finalsizethmu} implies that, for $0 < \varepsilon< \rho
^b$, $\limsup_{n \to\infty}\mathbb{P} (n^{-1}|\mathcal{R}^{(n)}|
\le\varepsilon ) \le1-\rho$,
which, together with \eqref{limpleneps1}, yields that, for such
$\varepsilon$,
\[
\lim_{n \to\infty}\mathbb{P} \bigl(n^{-1}\bigl|
\mathcal{R}^{(n)}\bigr| \le \varepsilon \bigr) =1-\rho.
\]
The theorem then follows upon combining this observation with \eqref
{subfinalsizecd1} and Theorem \ref{finalsizethmu}.
\end{pf}

\section{Properties of the forward branching process}
\label{forwardBP}
In this section we study the survival probability of the branching
process $\mathcal{Z}^f$ introduced in Section~\ref{earlystages}. Recall
that individuals in $\mathcal{Z}^f$ are typed by the length of the
infectious period of the corresponding individual in the epidemic process.
There is one ancestor, $i^*$ say, whose type is distributed as $\mathcal
{I}$ and who belongs to $X\sim\mathcal{MP}(A)$ cliques. [I.e., the
corresponding individual in the epidemic process $\mathcal
{E}^{(n)}(A,B,\mathcal{I})$ belongs to $X\sim\mathcal{MP}(A)$ cliques.]
Those cliques have sizes that are independent and identically
distributed as $1+\check{Y}$, where $\check{Y}\sim\mathcal{MP}(\tilde
{B})$. The offspring of the ancestor correspond to the individuals who,
in the corresponding epidemic process, are infected by the local
epidemics triggered by $i^*$ in the $X$ cliques it belongs to. The
offspring of $i^*$ are grouped into litters, with each litter
corresponding to a clique of~$i^*$. Note that some litters might be
empty (if the epidemic fails to spread further into some cliques to
which $i^*$ belongs).
The offspring of any subsequent individual is defined similarly, except
that such an individual belongs to $\check{X}\sim\mathcal{MP}(\tilde
{A})$ cliques in addition to the clique it was infected through.
The type space for $\mathcal{Z}^f$ is given by the support of $\mathcal
{I}$, which for ease of exposition we assume is $(0,\infty]$. Extension
to cases where $\mathcal{I}$ is supported on a proper subset of
$(0,\infty]$ is straightforward.

We investigate the survival probability of $\mathcal{Z}^f$ using
functionals defined on measurable test functions $h\dvtx (0,\infty] \to
[0,1]$ as follows (cf. \cite{Boll07,Boll10}). Let $h(x)$ be a given
test function. Suppose that individuals in $\mathcal{Z}^f$ are marked
independently, with an individual of type $x$ being marked with
probability $h(x)$. Let $F(h)(x)$ be the probability that an ancestor
of type $x$ has at least one marked child in a given litter and let
$\Phi(h)(x)$ be the probability that an ancestor of type $x$ has at
least one marked child. Recall that the probability generating function
of $X$ is given by $f_X(s)=\phi_A(1-s)$ ($s \in[0,1]$), where $\phi
_A(\theta) = \mathbb{E}[\mathrm{e}^{-\theta A}]$ is the moment generating
function of $A$. It follows that
\[
\Phi(h) (x)=1-\phi_A\bigl(F(h) (x)\bigr).
\]
Define the functional $\tilde{\Phi}(h)(x)$ similarly for the branching
process $\tilde{\mathcal{Z}}^f$, defined in the final paragraph of
Section~\ref{earlyStagesGeneralIP}; thus
\[
\tilde{\Phi}(h) (x)=1-\phi_{\tilde{A}}\bigl(F(h) (x)\bigr).
\]

Let $\rho_i$ be the probability that generation $i$ of the branching
process $\mathcal{Z}^f$ is nonempty, that is $\rho_i= \mathbb
{P}(|\mathcal{Z}^f_i| > 0)$. By definition $\rho_i$ is nonincreasing,
so $\rho=\lim_{i \to\infty} \rho_i$ exists and is the probability of
survival of the branching process.
Let $\tilde{\rho}_i(x)$ be the probability that the lineage of an
individual (i.e. the sub-process consisting of that individual and all
its descendants), which is not the ancestor and has type $x$, survives
for at least $i$ further generations and let $\tilde{\rho}(x) = \lim_{i\to\infty} \tilde{\rho}_i(x)$ be the probability that this lineage
survives forever. Note that $\tilde{\rho}_1(x) = \tilde{\Phi}(\mathbf
{1})(x)$, where $\mathbf{1}$ is the function which is equal to 1 on its
entire domain. It is clear that $\tilde{\rho}(x)$ satisfies
%
%
\begin{equation}
\label{survivetilde} \tilde{\rho}(x) = \tilde{\Phi}(\tilde{\rho}) (x),
\end{equation}
since in order for the lineage of an individual to survive, at least
one of the children of that individual must have a surviving lineage.
Furthermore,
%
%
\begin{equation}
\label{survive} \rho= \int_{(0,\infty]} \Phi(\tilde{\rho}) (x)
\mathbb{P}(\mathcal {I}\in\mathrm{d}x)= \mathbb{E}\bigl[\Phi(\tilde{\rho}) (
\mathcal{I})\bigr].
\end{equation}
Let $\tilde{\Phi}_i$ be the $i$th iterate of $\tilde{\Phi}$ and note
that $\tilde{\rho}_i(x)=\tilde{\Phi}_i(\mathbf{1})(x)$.
The functionals $\Phi(h)(x)$ and $\tilde{\Phi}(h)(x)$ are monotonic
increasing in $h$ (e.g., if $h_1(x) \ge h_2(x)$ for all $x \in(0,\infty
]$, then $\Phi(h_1)(x) \ge\Phi(h_2)(x)$ for all $x \in(0,\infty]$).
Therefore, $\tilde{\rho}(x) = \lim_{i \to\infty} \tilde{\Phi}_i(\mathbf
{1})(x)$ is the pointwise maximal solution of \eqref{survivetilde}.
Note that, since $\mathcal{Z}^f$ is irreducible, either $\tilde{\rho
}(x)=0$ for all $x\in(0,\infty]$ or $\tilde{\rho}(x)>0$ for all $x\in
(0,\infty]$. The following lemma is proved and discussed in Appendix
\ref{appendix1}.

\begin{lemma}\label{unilem}
There is at most one nonzero solution $\tilde{\rho}(x)$ of \eqref
{survivetilde}.
\end{lemma}

Now recall the definition of $R_*$ from \eqref{R*def}, where $\check
{Y}\sim\mathcal{MP}(\tilde{B})$, and as before let $T(m)$ denote the
size of a litter in a clique of $m$ initial susceptibles, in which the
infectious periods of infectives are i.i.d. copies of $\mathcal{I}$.
It is convenient here to show explicitly the dependence on $\mathcal
{I}$ and write $T(m)=T(m,\mathcal{I})$, so
\[
R_*=\mathbb{E}_{\check{Y}}\bigl[\mathbb{E}\bigl[T(\check{Y},\mathcal{I})\mid
\check {Y}\bigr]\bigr]\mathbb{E}[\tilde{A}].
\]

\begin{theorem}\label{survandR0}
The survival probability satisfies $\rho>0$ if and only if $R_*>1$.
\end{theorem}

\begin{pf}
Suppose first that $R_*>1$. For $k \in\mathbb{Z}_+$, let $L(k,\mathcal
{I})=\mathbb{E}[T(\check{Y},\mathcal{I})\mid\check{Y}=k]$. Then there
exists $K \in\mathbb{N}$ such that
%
%
\begin{equation}
\mathbb{E}[\tilde{A}] \sum_{k=0}^K L(k,
\mathcal{I}) \mathbb{P}(\check {Y}=k) >1. \label{survlem1}
\end{equation}

For $\varepsilon>0$, let $\mathcal{I}_{\varepsilon}$ be the discrete random
variable obtained from $\mathcal{I}$ by $\mathcal{I}_{\varepsilon}=
\varepsilon\lfloor\mathcal{I}/\varepsilon\rfloor$ (with the convention that
$\lfloor\infty\rfloor= \infty$) and note that $\mathcal{I}_{\varepsilon
}$ is stochastically smaller than $\mathcal{I}$.
Since $L(k,\mathcal{I})$ depends on the realisation of an epidemic
generated graph defined on a finite clique,
there exists $\varepsilon>0$ such that
\[
\mathbb{E}[\tilde{A}] \sum_{k=0}^K L(k,
\mathcal{I}_\varepsilon) \mathbb {P}(\check{Y}=k) >1.
\]
Analagously to the derivation of \eqref{survlem1}, there exists
$K'_{\varepsilon} \in\mathbb{N}$ such that for
$\mathcal{I}'_{\varepsilon}= \mathcal{I}_{\varepsilon} \ind(\mathcal
{I}_{\varepsilon} \notin(K'_{\varepsilon},\infty))$,
we have
%
%
\begin{equation}
\mathbb{E}[\tilde{A}] \sum_{k=0}^{K} L
\bigl(k,\mathcal{I}'_{\varepsilon}\bigr) \mathbb{P}(\check{Y}=k)
>1. \label{RstarIepsPrime}
\end{equation}

Consider the branching process $\tilde{\mathcal{Z}}^f(A,B,\mathcal
{I}'_{\varepsilon})$, which has finitely many types and is irreducible.
Let $\tilde{M}$ be the mean offspring matrix of $\tilde{\mathcal
{Z}}^f(A,B,\mathcal{I}'_{\varepsilon})$. Note that whether or not an
individual in a clique becomes infected is independent of that
individual's own infectious period. It follows that the rows of $\tilde
{M}$ are each proportional to the probability mass function of $\mathcal
{I}'_{\varepsilon}$, so $\tilde{M}$ has rank one and the maximal
eigenvalue of $\tilde{M}$ is given by its trace, which is easily seen
to be equal to the left-hand side of \eqref{RstarIepsPrime} with $K$
replaced by $\infty$. Therefore, if $R_*>1$, the branching process
$\tilde{\mathcal{Z}}^f(A,B,\mathcal{I})$ dominates the irreducible
finite-type supercritical branching process $\tilde{\mathcal
{Z}}^f(A,B,\mathcal{I}'_\varepsilon)$, which we know from standard
theory (Theorem~4.2.2 of \cite{Jage75}), has a strictly positive probability
of survival. Thus $\tilde{\rho}(x)>0$ for all $x\in(0,\infty]$;
equation \eqref{survive} then implies that $\rho>0$.

For $R_* \leq1$ we use a similar argument to \cite{Boll10}. Suppose
that $R_* \leq1$ and that $\tilde{\rho}(x)>0$ for some (and thus all)
$x\in(0,\infty]$. Recall that $\tilde{\Phi}(\tilde{\rho})(x)$ is the
probability that, in $\tilde{\mathcal{Z}}^f(A,B,\mathcal{I})$ and with
individuals of type $x$ being marked with probability $\tilde{\rho
}(x)$, an individual of type $x$ has at least one marked child. Note
that this probability is strictly smaller than the expectation of the
number, $T_M(x,\tilde{\rho})$ say, of marked children of such an
individual. Let $T(x,m,\mathcal{I})$ denote the size of a single-clique
epidemic with $m$ initial susceptibles and a single initial infective
which has infectious period $x$. Then, again exploiting the fact that
whether or not an individual is infected is independent of its
infectious period, we find that
\[
\mathbb{E}\bigl[T_M(x,\tilde{\rho})\bigr]=\mathbb{E}[\tilde{A}]
\mathbb{E}_{\check
{Y}}\bigl[\mathbb{E}\bigl[T(x,\check{Y},\mathcal{I})\mid
\check{Y}\bigr]\bigr]\mathbb{E}\bigl[\tilde {\rho}(\mathcal{I})\bigr],
\]
whence, recalling \eqref{survivetilde},
%
%
\begin{equation}
\label{tildephiub} \tilde{\rho}(x)=\tilde{\Phi}(\tilde{\rho}) (x)<\mathbb{E}[\tilde
{A}]\mathbb{E}_{\check{Y}}\bigl[\mathbb{E}\bigl[T(x,\check{Y},\mathcal{I})\mid
\check {Y}\bigr]\bigr]\mathbb{E}\bigl[\tilde{\rho}(\mathcal{I})\bigr].
\end{equation}
Note that if $x$ is a realisation of a random variable $\mathcal{I}_0$
that is distributed as~$\mathcal{I}$, then $\mathbb{E}[T(m,\mathcal
{I})]=\mathbb{E}_{\mathcal{I}_0}[\mathbb{E}[T(\mathcal{I}_0,m,\mathcal
{I})\mid\mathcal{I}_0]]$ and \eqref{tildephiub} implies that $\mathbb
{E}[\tilde{\rho}(\mathcal{I})] < R_* \mathbb{E}[\tilde{\rho}(\mathcal
{I})]$. It then follows that $R_*>1$, which is a contradiction. Thus,
if $R_* \leq1$ then $\tilde{\rho}(x)$ is identically zero on the
support of $\mathcal{I}$ and it then follows from \eqref{survive} that
$\rho=0$.
\end{pf}

\section{Proofs}
\label{proofs}
In this section we give formal proofs of Theorems \ref{extinctionthmu}
and \ref{finalsizethmu}. Recall the probability space $(\Omega,\mathcal
{F}, \nu)$ defined in Section~\ref{RIGs}, where $\Omega$ is the product
space of nonnegative real-valued infinite sequences $(A_i, i \in
\mathbb{N})$ and $(B_j, j \in\mathbb{N})$ and $\nu$ is the appropriate
(product) measure determined by the distributions of $A$ and $B$. In
the proofs we consider processes which depend on $\omega\in\Omega$,
that is, on the sequences $(A_i, i \in\mathbb{N})$ and $(B_i, i \in
\mathbb{N})$. The measure governing a process conditioned on $\omega$
is denoted by $\mathbb{P}_{\omega}$ and the corresponding expectation
by $\mathbb{E}_\omega$.
We use the notation $X_n \parrow X$ to denote that $X_n$ converges in
probability to $X$ as $n \to\infty$, with respect to the measure $\nu$.
That is, $X_n \parrow X$ means that for every $\varepsilon>0$, $\delta>0$,
we have $\nu(|X_n-X|> \varepsilon) < \delta$ for all sufficiently large $n
\in\mathbb{N}$. In particular, we often use the notation $\mathbb
{P}_{\omega}(X_n \in\mathcal{A}) \parrow\mathbb{P}(X \in\mathcal
{A})$, which is to be interpreted as meaning that, for a subset
$\mathcal{A}$ of the state space of $X_n$ and $X$, we have that for
every $\varepsilon>0$,
\[
\int_{\omega\in\Omega} \ind\bigl(\bigl|\mathbb{P}_{\omega}(X_n
\in\mathcal {A})-\mathbb{P}(X \in\mathcal{A})\bigr|>\varepsilon\bigr) \nu(\mathrm{d}
\omega) \to0 \qquad\mbox{as $n \to\infty$}.
\]
Additionally, we use the notation $X_n= O_{p_{\nu}}(g(n))$ if there
exists a constant $C>0$ such that
$\mathbb{P}_{\omega} (X_n < C g(n)) \parrow1$ and $X_n= \Theta_{p_{\nu
}}(g(n))$ if there exist constants $0<c<C$ such that
$\mathbb{P}_{\omega} (c g(n) < X_n < C g(n)) \parrow1$.

We prove the following conditioned versions of Theorems \ref{extinctionthmu} and \ref{finalsizethmu}, in which $\mathcal
{R}^{(n)}(\omega,\mathcal{I})$ denotes the set of ultimately recovered
vertices, including the single initial infective, in an SIR epidemic
(as defined in Section~\ref{SIRepidemics}) on the random intersection
graph $G^{(n)}$, constructed using the infectious period distribution
$\mathcal{I}$ and the sequences $(A_i, i \in\mathbb{N})$, $(B_j, j \in
\mathbb{N})$ denoted by $\omega\in\Omega$.

\begin{theorem}\label{extinctionthm}
For $k \in\mathbb{N}$, we have
\[
\mathbb{P}_{\omega}\bigl(\bigl|\mathcal{R}^{(n)}(\omega,
\mathcal{I})\bigr|=k\bigr) \parrow \mathbb{P}\bigl(\bigl|\mathcal{Z}^f(A,B,
\mathcal{I})\bigr|=k\bigr).
\]
\end{theorem}

\begin{theorem}\label{finalsizethm}
Suppose that $R_*>1$. Then for every $0 < \varepsilon< \rho^b(A,B,\mathcal
{I})$, 
\[
\mathbb{P}_{\omega} \bigl( \bigl\llvert n^{-1}\bigl|
\mathcal{R}^{(n)}(\omega,\mathcal {I})\bigr| - \rho^b(A,B,
\mathcal{I})\bigr\rrvert < \varepsilon \bigr) \parrow\rho (A,B,\mathcal{I}).
\]
\end{theorem}

\begin{pf*}{Proofs of Theorems \ref{extinctionthmu} and \ref{finalsizethmu}}
Note that, for fixed $k \in\mathbb{N}$, the sequence of random variables
$(\mathbb{P}_{\omega}(|\mathcal{R}^{(n)}(\omega,\mathcal{I})|=k), n \in
\mathbb{N})$ is uniformly integrable, so Theorem \ref{extinctionthmu}
follows immediately from Theorem \ref{extinctionthm}
(and Theorem
7.10.3 of~\cite{Grim92}), by taking expectations with respect to the measure
$\nu$. Theorem \ref{finalsizethmu} follows similarly from Theorem \ref{finalsizethm}.
\end{pf*}

\subsection{Proof of Theorem \protect\texorpdfstring{\ref{extinctionthm}}{5.1}}
\label{pextinctionthm}

In this proof we use three processes:
\begin{itemize}
\item the branching process $\mathcal{Z}^f = \mathcal{Z}^f(A,B,\mathcal{I})$,
\item the branching process $\mathcal{Z}^{(n)} = \mathcal
{Z}^f(A^{(n)},B^{(n)},\mathcal{I})$, defined similarly to $\mathcal
{Z}^f(A,\allowbreak B,\mathcal{I})$ but with $A$ and $B$ replaced, respectively, by
$A^{(n)}$ and $B^{(n)}$, defined in \eqref{AnDef} and \eqref{BnDef},
\item the exploration process of the epidemic generated graph on
$G^{(n)}$, denoted by $\mathcal{R}^{(n)} = \mathcal{R}^{(n)}(\omega
,\mathcal{I}) = (\mathcal{R}^{(n)}_0,\mathcal{R}^{(n)}_1,\ldots)$.
\end{itemize}
%
In the exploration process, $\mathcal{R}^{(n)}_0$ denotes the initially
infective vertex $v_0$, $\mathcal{R}^{(n)}_1$ denotes the subset of
vertices in $V^{(n)}\setminus\mathcal{E}^{(n)}_0$ that in the epidemic
generated graph have an edge to them from $v_0$,
$\mathcal{R}^{(n)}_2$ denotes the subset of vertices in
$V^{(n)}\setminus(\mathcal{R}^{(n)}_0 \cup\mathcal{R}^{(n)}_1)$ that
in the epidemic generated graph have an edge to them from at least one
member of $\mathcal{R}^{(n)}_1$, and so on. With slight abuse of
notation we now use $\mathcal{R}^{(n)}$ for the exploration process,
where previously it was the set of ultimately recovered vertices in
$\mathcal{E}^{(n)}$. As with the branching process $\mathcal{Z}^f$,
$|\mathcal{R}^{(n)}|=\sum_{i=0}^{\infty}|\mathcal{R}^{(n)}_i|$ is the
total number of ultimately recovered vertices; note that this has
precisely the same meaning as in Section~\ref{SIRRIG}.

To prove Theorem \ref{extinctionthm} we first show that the
distribution of the total size of $\mathcal{Z}^{(n)}$ is approximately
that of $\mathcal{Z}^f$, then that the distribution of the total size
of $\mathcal{R}^{(n)}$ is approximately that of $\mathcal{Z}^{(n)}$.

\begin{lemma}\label{bplimits}
For $k \in\mathbb{N}$, it holds that $\mathbb{P}_\omega(|\mathcal
{Z}^{(n)}| = k) \parrow\mathbb{P}(|\mathcal{Z}^f| = k)$.
\end{lemma}

\begin{pf}
Recall that a litter in a branching process is a group of children
corresponding to the individuals infected in a local
epidemic in one clique, excluding the initial susceptible.
Let the total number of (possibly empty) litters in $\mathcal{Z}^f$ and
$\mathcal{Z}^{(n)}$ be denoted by $H$ and $H^{(n)}$, respectively.
Note that, using Theorem 7.2.19 of \cite{Grim92}, if $X_n \Rightarrow X$,
then $\mathcal{MP}(X_n) \Rightarrow\mathcal{MP}(X)$.
Recall further that $A^{(n)} \Rightarrow A$ and $B^{(n)} \Rightarrow B$
as $n \to\infty$. These latter convergence results also hold for the
size-biased variants, as is shown just below equation \eqref{sibidi}.
It follows that, as $n \to\infty$, the number and sizes of litters
spawned by a typical individual in $\mathcal{Z}^{(n)}$ converge in
distribution to those of a corresponding typical individual in $\mathcal{Z}^f$.
Hence, for $k \in\mathbb{N}$ and $l \in\mathbb{Z}_+$,
\[
\mathbb{P}_{\omega}\bigl(\bigl|\mathcal{Z}^{(n)}\bigr| = k,
H^{(n)}=l\bigr) \parrow\mathbb {P}\bigl(\bigl|\mathcal{Z}^f\bigr| = k,
H=l\bigr).
\]
Therefore, for every $l \in\mathbb{N}$, we have
%
%
\begin{equation}
\label{zhnconv} \mathbb{P}_{\omega}\bigl(\bigl|\mathcal{Z}^{(n)}\bigr| = k,
H^{(n)} \leq l\bigr) \parrow \mathbb{P}\bigl(\bigl|\mathcal{Z}^f\bigr|
= k, H \leq l\bigr).
\end{equation}
Note that
%
%
\begin{equation}
\label{znequ}\qquad \mathbb{P}_{\omega}\bigl(\bigl|\mathcal{Z}^{(n)}\bigr| = k
\bigr) = \mathbb{P}_{\omega
}\bigl(\bigl|\mathcal{Z}^{(n)}\bigr| = k,
H^{(n)} \leq l\bigr) + \mathbb{P}_{\omega}\bigl(\bigl|
\mathcal{Z}^{(n)}\bigr| = k, H^{(n)} > l\bigr)
\end{equation}
and
\[
\mathbb{P}\bigl(\bigl|\mathcal{Z}^f\bigr| = k\bigr) = \mathbb{P}\bigl(\bigl|
\mathcal{Z}^f\bigr| = k, H \leq l\bigr) + \mathbb{P}\bigl(\bigl|
\mathcal{Z}^f\bigr| = k, H > l\bigr).
\]

Fix $k \in\mathbb{N}$ and $\varepsilon>0$. Let $H_k$ be the total number
of litters spawned by the first $k$ vertices evaluated in the branching
process $\mathcal{Z}^f$, so $H_k$ is distributed as $X_0+\check
{X}_1+\check{X}_2+\cdots+\check{X}_{k-1}$, where $X_0,\check{X}_1,\check
{X}_2,\ldots,\check{X}_{k-1}$ are independent, $X_0\sim\mathcal{MP}(A)$
and $\check{X}_i\sim\mathcal{MP}(\tilde{A}) $ $(i=1,2,\ldots,k-1)$.
Now, for any $l \in \mathbb{Z}_+$,
\[
\mathbb{P}\bigl(\bigl|\mathcal{Z}^f\bigr| = k, H > l\bigr)=\mathbb{P}\bigl(\bigl|
\mathcal{Z}^f\bigr| = k, H_k > l\bigr) \le\mathbb{P}(
H_k > l).
\]
Further, $H_k$ is a proper random variable, so $\mathbb{P}( H_k > l)
\downarrow0$ as $l \to\infty$. Thus, there exists $L =L(k, \varepsilon)
\in\mathbb{N}$, such that for all $l>L$,
%
%
\begin{equation}
\label{zinqeu1} \mathbb{P}\bigl(\bigl|\mathcal{Z}^f\bigr| = k\bigr) <
\mathbb{P}\bigl(\bigl|\mathcal{Z}^f\bigr| = k,H \leq l\bigr) + \varepsilon/3.
\end{equation}

Let $H_k^{(n)}$ be the total number of litters in the first $k$
vertices evaluated in the branching process $\mathcal{Z}^{(n)}$. Then,
$H_k^{(n)}\Rightarrow H_k$, since $\mathcal{MP}(A^{(n)}) \Rightarrow
\mathcal{MP}(A)$ and $\mathcal{MP}(\tilde{A}^{(n)}) \Rightarrow\mathcal
{MP}(\tilde{A})$, whence
$\mathbb{P}_{\omega}(H^{(n)}>l) \parrow\mathbb{P}(H>l)$,
for any $l \in \mathbb{Z}_+$. Arguing similarly to above and
recalling \eqref{znequ}, it follows that, given any $\delta>0$, there
exists $L' =L'(k, \varepsilon, \delta) \in\mathbb{N}$, such that for all $l>L'$,
%
%
\begin{equation}
\label{zinqeu2} \nu \bigl(\mathbb{P}_{\omega}\bigl(\bigl|\mathcal{Z}^{(n)}\bigr|
= k\bigr) < \mathbb {P}_{\omega}\bigl(\bigl|\mathcal{Z}^{(n)}\bigr| =
k,H^{(n)} \leq l\bigr) + \varepsilon/3 \bigr)>1-\delta/2
\end{equation}
for all sufficiently large $n$.

Fix $\delta>0$ and choose $l > \max(L,L')$. Then \eqref{zhnconv}
implies that
%
%
\begin{equation}
\label{zinqeu3} \nu \bigl( \bigl|\mathbb{P}_{\omega}\bigl(\bigl|
\mathcal{Z}^{(n)}\bigr| = k,H^{(n)} \leq l\bigr) - \mathbb{P}\bigl(\bigl|
\mathcal{Z}^f\bigr| = k,H \leq l\bigr) \bigr|<\varepsilon/3 \bigr) > 1-\delta/2
\end{equation}
for all sufficiently large $n$.
Using the triangle inequality, we obtain
\begin{eqnarray*}
&&\bigl\llvert \mathbb{P}_{\omega} \bigl(\bigl|\mathcal{Z}^{(n)}\bigr| = k
\bigr)- \mathbb {P}\bigl(\bigl|\mathcal{Z}^f\bigr| = k\bigr)\bigr\rrvert
\\
&&\qquad\leq
\bigl\llvert \mathbb{P}_{\omega}\bigl(\bigl|\mathcal{Z}^{(n)}\bigr| = k
\bigr) - \mathbb {P}_{\omega}\bigl(\bigl|\mathcal{Z}^{(n)}\bigr| =
k,H^{(n)} \leq l\bigr)\bigr\rrvert
\\
& &\qquad\quad{}+\bigl\llvert \mathbb{P}_{\omega}\bigl(\bigl|\mathcal{Z}^{(n)}\bigr| =
k,H^{(n)} \leq l\bigr) - \mathbb{P}\bigl(\bigl|\mathcal{Z}^f\bigr| =
k,H \leq l\bigr)\bigr\rrvert
\\
& &\qquad\quad{}+\bigl\llvert \mathbb{P}\bigl(\bigl|\mathcal{Z}^f\bigr| = k\bigr) -
\mathbb{P}\bigl(\bigl|\mathcal{Z}^f\bigr| = k,H \leq l\bigr)\bigr\rrvert,
\end{eqnarray*}
whence, noting that the final term is independent of $\omega$,
\begin{eqnarray*}
&&\nu \bigl( \bigl|\mathbb{P}_{\omega}\bigl(\bigl|\mathcal{Z}^{(n)}\bigr| = k
\bigr) -  \mathbb {P}\bigl(\bigl|\mathcal{Z}^f\bigr| = k\bigr) \bigr| \geq
\varepsilon \bigr)
\\
&&\qquad\leq \nu \bigl(\mathbb{P}_{\omega}\bigl(\bigl|\mathcal{Z}^{(n)}\bigr| =
k\bigr) \geq \mathbb{P}_{\omega}\bigl(\bigl|\mathcal{Z}^{(n)}\bigr| =
k,H^{(n)} \leq l\bigr) + \varepsilon /3 \bigr)
\\
&&\qquad\quad{} + \nu \bigl( |\mathbb{P}_{\omega}\bigl(\bigl|\mathcal{Z}^{(n)}\bigr| =
k,H^{(n)} \leq l\bigr) - \mathbb{P}\bigl(\bigl|\mathcal{Z}^f\bigr| =
k,H \leq l\bigr) | \geq\varepsilon /3 \bigr)
\\
&&\qquad\quad{} + \ind \bigl(\mathbb{P}\bigl(\bigl|\mathcal{Z}^f\bigr| = k\bigr) \geq
\mathbb{P}\bigl(\bigl|\mathcal {Z}^f\bigr| = k,H \leq l\bigr) + \varepsilon/3
\bigr).
\end{eqnarray*}
By choosing $l$ large enough, it follows, using \eqref{zinqeu1}, \eqref
{zinqeu2} and \eqref{zinqeu3}, that for all sufficiently large $n$,
\[
\nu \bigl( \bigl|\mathbb{P}_{\omega}\bigl(\bigl|\mathcal{Z}^{(n)}\bigr| = k
\bigr)- \mathbb {P}\bigl(\bigl|\mathcal{Z}^f\bigr| = k\bigr)\bigr | \geq\varepsilon
\bigr) \leq\delta/2 + \delta /2 + 0 = \delta,
\]
and the lemma then follows.
\end{pf}

\begin{lemma}\label{fcoupling} For $k\in\mathbb{N}$,
$\mathbb{P}_{\omega}(|\mathcal{Z}^{(n)}| \leq k) - \mathbb{P}_{\omega
}(|\mathcal{R}^{(n)}| \leq k) \parrow0$.
\end{lemma}
\begin{pf} 
The proof follows from a standard coupling argument, described below.
Firstly though, for each $n\in\mathbb{N}$, let $v^{(n)}_0$ be a vertex
chosen uniformly at random from $V^{(n)}$, and let
$v^{(n)}_1,v^{(n)}_2, \ldots $ be independently chosen vertices from
$V^{(n)}$, where the probability that a given vertex is chosen is
proportional to its $A$-weight. Let $a^{(n)}_0,a^{(n)}_1, \ldots $ be
the respective $A$-weights of $v^{(n)}_0, v^{(n)}_1, \ldots.$ Let
$\mathcal{I}^{(n)}_0$ be the type assigned to vertex $v^{(n)}_0$.
Let $v'^{(n)}_1,v'^{(n)}_2, \ldots $ be independently chosen vertices
(representing cliques) from $V'^{(n)}$ where the probability that a
given vertex is chosen is proportional to its $B$-weight. The
$B$-weights of $v'^{(n)}_1,v'^{(n)}_2, \ldots $ are denoted by
$b^{(n)}_1,b^{(n)}_2, \ldots,$ respectively.
Let the random variable
\[
T^{(n)} = \min\bigl(i\in\mathbb{N}\dvtx v^{(n)}_i
= v^{(n)}_j \mbox{ for some $j<i$}\bigr)
\]
be the smallest index at which a vertex from $V^{(n)}$ is chosen a
second time. Similarly, define
\[
T'^{(n)} = \min\bigl(i \in\mathbb{N}\dvtx
v'^{(n)}_i = v'^{(n)}_j
\mbox{ for some $j<i$}\bigr).
\]

The constructions of $\mathcal{Z}^{(n)}$ and $\mathcal{R}^{(n)}$ are
coupled as follows. The ancestor of $\mathcal{Z}^{(n)}$ spawns a
$\Poisson(a^{(n)}_0)$ number of (possibly empty) litters, $l'$ say. The
cliques that the initial infective in $\mathcal{R}^{(n)}$ belongs to
are given by
$v'^{(n)}_1, v'^{(n)}_2, \ldots, v'^{(n)}_{l'}$,
which might contain duplicates; the $B$-weights associated with these
litters are $b^{(n)}_1,b^{(n)}_2, \ldots, b^{(n)}_{l'}$. If
$T'^{(n)}>l'$, then there are no duplicates amongst $v'^{(n)}_1,
v'^{(n)}_2, \ldots, v'^{(n)}_{l'}$, and the processes stay coupled. If
not, the construction can be continued, but the details are not
important for our purposes.

If the coupling continues, the sizes of the litters (recall that litters
are defined both for the epidemic process and the branching process)
are then determined.
For each $i=1,2,\ldots,l'$, the size of litter $i$ is distributed as
the number of initially susceptible individuals which are
infected during a local epidemic in a group with one initially infectious
individual, having infectious period $\mathcal{I}^{(n)}_0$, and a
$\mathcal{P}(b^{(n)}_i)$ distributed number of initially susceptible
individuals. The litter sizes are all independent. Say that the total
number of vertices in the $l'$ litters is $l$, then they get
$A$-weights $a^{(n)}_1,a^{(n)}_2, \ldots, a^{(n)}_l$ and types $\mathcal
{I}^{(n)}_1,\mathcal{I}^{(n)}_2, \ldots, \mathcal{I}^{(n)}_l$, which
are i.i.d. and distributed as $\mathcal{I}$.
If $l < T^{(n)}$ the coupling continues, and the generation 1 vertices
are $v^{(n)}_1,v^{(n)}_2, \ldots, v^{(n)}_l$.
The coupling now proceeds in the obvious way. Note that in this
construction we have not yet decided which vertices are in the same
clique (of the random intersection graph) as $v^{(n)}_1$ but are not
infected by the local epidemic.

Let $H^{(n)}$ be as in the proof of Lemma \ref{bplimits}, and let
$H^{(*n)}$ be the corresponding number for $\mathcal{R}^{(n)}$. We need
to prove that for $k\in\mathbb{N}$ and $l \in\mathbb{Z}_+$,
\[
\mathbb{P}_{\omega}\bigl(\bigl|\mathcal{Z}^{(n)}\bigr| = k,
H^{(n)}=l\bigr) - \mathbb {P}_{\omega}\bigl(\bigl|
\mathcal{R}^{(n)}\bigr| = k, H^{(*n)}=l\bigr) \parrow0,
\]
and then deduce the statement of the lemma as in the latter part of the
proof of Lemma \ref{bplimits}. Note that the coupling gives
%
%
\begin{eqnarray}\label{bpcouple}
&&\mathbb{P}_{\omega}\bigl(\bigl|\mathcal{Z}^{(n)}\bigr| = k,
H^{(n)}=l,T^{(n)}>k, T'^{(n)}>l\bigr)
\nonumber
\\[-8pt]
\\[-8pt]
\nonumber
&&\qquad = \mathbb{P}_{\omega}\bigl(\bigl|\mathcal{R}^{(n)}\bigr| = k,
H^{(*n)}=l,T^{(n)}>k, T'^{(n)}>l\bigr).
\end{eqnarray}
Furthermore, letting $C^{(n)}(k,l) = \{T^{(n)} \leq k\} \cup\{T'^{(n)}
\leq l\}$, we have
\begin{eqnarray*}
\mathbb{P}_{\omega}\bigl(\bigl|\mathcal{Z}^{(n)}\bigr| = k,
H^{(n)}=l\bigr) &= & \mathbb{P}_{\omega}\bigl(\bigl|
\mathcal{Z}^{(n)}\bigr| = k, H^{(n)}=l,T^{(n)}>k,
T'^{(n)}>l\bigr)
\\
&&{} + \mathbb{P}_{\omega}\bigl(\bigl|\mathcal{Z}^{(n)}\bigr| = k,
H^{(n)}=l,C^{(n)}(k,l)\bigr).
\end{eqnarray*}
Note that the second term on the right-hand side of this expression is
bounded above by $\mathbb{P}_{\omega}(C^{(n)}(k,l))$.

Recall from Section~\ref{RIGs} that $\mu=\mathbb{E}[A]=\alpha\mathbb
{E}[B]<\infty$, which implies that the total weight of vertices in
$V^{(n)}$ with weight exceeding $\log n$ is $\nu$-almost surely~$o(n)$.
(To show this, note that, since $\mu< \infty$, for any $N>0$,
\[
n^{-1}\sum_{i=1}^n
A_i \ind(A_i>N) \asarrows \mathbb{E}\bigl[A\ind(A>N)\bigr]\qquad
\mbox{as $n \to\infty$}
\]
and $\mathbb{E}[A\ind(A>N)] \to0$ as $N \to\infty$.) A similar result
holds for the weights of the vertices in $V'^{(n)}$.
Hence, for every $k,l \in\mathbb{N}$, the probability that both $\max
(a^{(n)}_i \dvtx 0 \leq i \leq k) \leq\log n$ and $\max(b^{(n)}_j \dvtx 1 \leq
j \leq l) \leq\log n$ converges to 1 as $n\to\infty$. Thus, the total
weight of the first $k$ vertices and the first $l$ litters chosen in
the branching process is $\nu$-almost surely $O(\log n)$. By a birthday
problem argument we deduce that $\mathbb{P}_{\omega}(C^{(n)}(l,k))
\parrow0$. (Note that if $M_n(k)$ is the number of distinct pairs
$(i,j)$ with $0\leq i < j \leq k$ and $v_i^{(n)} = v_j^{(n)}$, then
under the above restrictions, $\mathbb{E}_\omega[M_n(k)]\leq\frac
{k(k-1)}{2}\frac{\log n}{L^{(n)}} \parrow0$).
Thus, for every $k,l\in\mathbb{N}$,
\[
\mathbb{P}_{\omega}\bigl(\bigl|\mathcal{Z}^{(n)}\bigr| = k,
H^{(n)}=l\bigr) - \mathbb{P}_{\omega}\bigl(\bigl|\mathcal{Z}^{(n)}\bigr|
= k, H^{(n)}=l,T^{(n)}>k, T'^{(n)}>l
\bigr)
\parrow0.
\]
Similarly, we deduce that, again for all $k,l\in\mathbb{N}$,
\begin{eqnarray*}
&&\mathbb{P}_{\omega}\bigl(\bigl|\mathcal{R}^{(n)}\bigr| = k,
H^{(*n)}=l\bigr) - \mathbb{P}_{\omega}\bigl(\bigl|\mathcal{R}^{(n)}\bigr|
= k, H^{(*n)}=l,T^{(n)}>k, T'^{(n)}>l
\bigr)\\
&&\qquad\parrow0,
\end{eqnarray*}
which, together with \eqref{bpcouple}, yields the lemma.
\end{pf}

Theorem \ref{extinctionthm} follows immediately by combining Lemmas \ref{bplimits} and \ref{fcoupling}.

\subsection{Proof of Theorem \protect\texorpdfstring{\ref{finalsizethm}}{5.2}}\label{pfinalsizethm}

Before considering susceptibility sets and backward branching
processes, we prove the following extension of Lemma \ref{bplimits}
which is required later in this section.
%
\begin{lemma}\label{survfin}
$\rho(A^{(n)},B^{(n)},\mathcal{I}) \parrow\rho(A,B,\mathcal{I})$.
\end{lemma}
\begin{pf}
For every $k\in\mathbb{Z}_+$, define the random variable
\[
\mathcal{I}^k (\mathcal{I})= \cases{ 2^{-k} \bigl
\lfloor2^k \mathcal{I} \bigr\rfloor,&\quad  $\mbox{if } \mathcal{I} <
2^k$, \vspace*{2pt}
\cr
2^k, &\quad $\mbox{if } \mathcal{I}
\in[2^k,\infty)$, \vspace*{2pt}
\cr
\infty,& \quad$\mbox{if } \mathcal{I} =
\infty.$}
\]
That is, $\mathcal{I}^k$ is a random variable which can take only
finitely many values and for $j=0,1,\ldots,4^k-1$,
\[
\mathbb{P}\bigl(\mathcal{I}^k = j 2^{-k}\bigr) =
\mathbb{P} \bigl(\mathcal{I} \in \bigl[j 2^{-k},(j+1) 2^{-k}\bigr)
\bigr),
\]
while
$\mathbb{P}(\mathcal{I}^k = 2^{k}) = \mathbb{P}(\mathcal{I} \in[2^k,
\infty))$ and $\mathbb{P}(\mathcal{I}^k = \infty) = \mathbb{P}(\mathcal
{I} = \infty)$.
It is clear that $\mathcal{I}^k \Rightarrow\mathcal{I}$ as $k \to
\infty$ and that $\mathcal{I}^k$ is stochastically smaller than
$\mathcal{I}^{k+1}$ for all $k \in\mathbb{Z}_+$.

For nonnegative random variables $X$ and $Y$, the function $\tilde{\rho
}(X,Y,\mathcal{I}^k)$ is pointwise nondecreasing in $k$, since it is
the survival probability of a branching process and (stochastically)
increasing the distribution of the infectious periods, and thus also of
the offspring distribution, cannot decrease the survival probability of
the process. By monotonicity we have that $\lim_{k \to\infty} \tilde
{\rho}(X,Y,\mathcal{I}^k)$ exists pointwise, and by the monotone
convergence theorem this limit satisfies \eqref{survivetilde} for
$\tilde{\rho}(X,Y,\mathcal{I})$.
By Lemma \ref{bplimits} we know that for every $k \in\mathbb{N}$,
$\mathbb{P}_{\omega}(|\mathcal{Z}^{(n)}| > k) \parrow\mathbb
{P}(|\mathcal{Z}^f| > k)$. This implies that for every $\varepsilon>0$ and
$\delta>0$, there exists $N_0 \in\mathbb{N}$ such that for $n>N_0$, we
have 
%
%
\begin{equation}
\label{rholims} \nu \bigl(\rho\bigl(A^{(n)},B^{(n)},\mathcal{I}
\bigr) < \rho(A,B,\mathcal{I}) + \varepsilon \bigr) > 1-\delta/2.
\end{equation}
Furthermore,
for every $\varepsilon>0$, there exists $K \in\mathbb{N}$ such that for
$k>K$, we have
\[
\rho\bigl(A,B,\mathcal{I}^k\bigr) > \rho(A,B,\mathcal{I}) -
\varepsilon/2.
\]
Similarly,
for every $\varepsilon>0$, $\delta>0$ and $k \in\mathbb{N}$, there exist
$N_k \in\mathbb{N}$ such that for $n>N_k$, we have
\[
\nu \bigl(\rho\bigl(A^{(n)},B^{(n)},\mathcal{I}^k
\bigr) > \rho\bigl(A,B,\mathcal{I}^k\bigr) - \varepsilon/2 \bigr) > 1-
\delta/2,
\]
while for every $k \in\mathbb{N}$ (and $\omega\in\Omega$), $\rho
(A^{(n)},B^{(n)},\mathcal{I}) \geq\rho(A^{(n)},B^{(n)},\mathcal
{I}^k)$. Combining these statements establishes that,
for every $\varepsilon>0$ and $\delta>0$, there exists $N\in\mathbb{N}$
such that for all $n>N$, we have
\[
\nu \bigl(\rho\bigl(A^{(n)},B^{(n)},\mathcal{I}\bigr) >
\rho(A,B,\mathcal{I}) - \varepsilon \bigr) > 1-\delta/2.
\]
Combining this with \eqref{rholims} completes the proof of the lemma.
\end{pf}

In order to prove Theorem \ref{finalsizethm}, we investigate the
susceptibility sets of two vertices chosen uniformly at random in the
subgraph $\hat{G}^{(n)}$ (of $G^{(n)}$), which is defined as follows.
Let $\hat{\mathbb{A}}^{(n)}$ be constructed from $\mathbb{A}^{(n)}$ by
ignoring all vertices in $V^{(n)}$ and $V'^{(n)}$ that have weights
larger than\vadjust{\goodbreak} $\log n$ and ignoring all edges that are incident to such
vertices. The graph $\hat{G}^{(n)}$ is constructed from $\hat{\mathbb
{A}}^{(n)}$ in the same way that $G^{(n)}$ is constructed from $\mathbb
{A}^{(n)}$.

We can create a realisation of $\hat{\mathbb{A}}^{(n)}$ as follows.
Define the vertex sets $\hat{V}^{(n)} = (v_i \in V^{(n)}\dvtx A_i \leq\log
n)$ and $\hat{V}'^{(n)} = (v'_j \in V'^{(n)}\dvtx B_j \leq\log n)$.
Conditional upon the weights of the vertices in $\mathbb{A}^{(n)}$, (i)
vertices $v_i \in\hat{V}^{(n)}$ and $v'_j \in\hat{V}'^{(n)}$ share in
$\hat{\mathbb{A}}^{(n)}$ a $\Poisson(A_iB_j/(\mu n))$ number of edges,
and (ii) the number of edges between distinct pairs of vertices are
independent. Let
\begin{eqnarray*}
\hat{L}^{(n)} & = &\sum_{i\dvtx v_i \in\hat{V}^{(n)}}
A_i \quad\mbox{and}
\\
\hat{L}'^{(n)} & =& \sum_{j:v'_j \in\hat{V}'^{(n)}}
B_j.
\end{eqnarray*}
Then the degree of vertex $v_i \in\hat{V}^{(n)}$ in $\hat{\mathbb
{A}}^{(n)}$ is $\Poisson(A_i \hat{L}'^{(n)}/(\mu n))$, and the degree
of $v'_j \in\hat{V}'^{(n)}$ is $\Poisson(B_j \hat{L}^{(n)}/(\mu n))$.
We construct from $\hat{\mathbb{A}}^{(n)}$ an identically distributed
copy of $\mathbb{A}^{(n)}$ by adding the vertices from $V^{(n)}
\setminus\hat{V}^{(n)}$ and
$V'^{(n)} \setminus\hat{V}'^{(n)}$ and, if $v_i \in V^{(n)}$ and $v'_j
\in V'^{(n)}$ are not both in $\hat{\mathbb{A}}^{(n)}$, letting $v_i$
and $v'_j$ share a $\Poisson(A_iB_j/(\mu n))$ number of newly-added
edges, independently of the number of edges between other vertices.

We construct a coupling of two independent branching processes and the
susceptibility sets of $v_1$ and $v_2$ in $\hat{G}^{(n)}$ (which by
exchangeability is equivalent to choosing two distinct vertices
uniformly at random), assuming that $A_1,A_2 \leq\log n $. We
therefore define [cf. equations \eqref{intermedA}--\eqref{BnDef}] $\hat
{A}^{(n)}_i = A_i   \ind(A_i \leq\log n) \hat{L}'^{(n)}/(\mu n)$ and
$\hat{B}^{(n)}_i = B_i   \ind(B_i \leq\log n) \hat{L}^{(n)}/(\mu n)$,
and let\break  $\hat{c}_A^{(n)} = \sum_{i=1}^n \ind(A_i \leq\log n)$ and $\hat
{c}_B^{(n)} = \sum_{i=1}^{\lfloor\alpha n\rfloor} \ind(B_i \leq\log
n)$. The random variables $\hat{A}^{(n)}$ and $\hat{B}^{(n)}$ are
defined by
\begin{eqnarray*}
\mathbb{P}_\omega\bigl(\hat{A}^{(n)} \leq x\bigr) & =& \bigl|\bigl
\{1 \leq i \leq\hat {c}_A^{(n)}\dvtx \hat{A}_i^{(n)}
\leq x\bigr\}\bigr| / \hat{c}_A^{(n)} \qquad(x\geq 0)\quad \mbox{and}
\\
\mathbb{P}_\omega\bigl(\hat{B}^{(n)} \leq x\bigr) & =& \bigl|\bigl
\{1 \leq i \leq\hat {c}_B^{(n)}\dvtx \hat{B}_i^{(n)}
\leq x\bigr\}\bigr| / \hat{c}_B^{(n)}\qquad (x\geq0).
\end{eqnarray*}
The processes through which the construction of the susceptibility set
of $v_i$ ($i \in\{1,2\}$) takes place are denoted by
\[
\hat{\mathcal{S}}^i = \hat{\mathcal{S}}^i\bigl(
\hat{A}^{(n)},\hat {B}^{(n)},\mathcal{I}\bigr) = \bigl(\hat{
\mathcal{S}}^i_j, j\in\mathbb{Z}_+\bigr).
\]
The two independent branching processes are $\mathcal{Z}^{b,i} =
\mathcal{Z}^{b,i}(\hat{A}^{(n)},\hat{B}^{(n)},\mathcal{I})$, for $i \in
\{1,2\}$, where $\hat{A}^{(n)}$ and $\hat{B}^{(n)}$ are as above.
The corresponding susceptibility set processes in $G^{(n)}$ are denoted
by $\mathcal{S}^i$ for $i \in\{1,2\}$. When no confusion is possible,
we sometimes suppress the reference to the starting vertex $i$.

We compute the probability that the susceptibility sets of two vertices
in $\hat{G}^{(n)}$ survive until at least generation
\[
t_n = \lceil\log\log n\rceil.
\]
We show that, with probability tending to $1$ as $n \to\infty$, if it
survives, the total number of individuals in the branching process
$\mathcal{Z}^{b}(\hat{A}^{(n)},\hat{B}^{(n)},\mathcal{I})$ in
generations $0,1,\ldots,t_n$ is of order $O(n^\varepsilon)$, for any
$\varepsilon> 0$. Then a standard coupling argument shows that, again
with probability tending to $1$ as $n \to\infty$, the susceptibility
process $\hat{\mathcal{S}}$ and its approximating branching process
$\mathcal{Z}^{b}(\hat{A}^{(n)},\hat{B}^{(n)},\mathcal{I})$ coincide
over generations $0,1,\ldots,t_n$; see the start of the proof of
Lemma \ref{sizesus}, which shows that for
large $n$, if the susceptibility set process survives until generation
$t_n$, its size will then be of order $O((\log n)^c)$, for some $c>0$.

Next, we show that, given any $\varepsilon>0$, there exists $K \in\mathbb
{N}$ such that the probability that both the $t_n$th generation of an
individual's susceptibility set is empty on $\hat{G}^{(n)}$ \emph{and}
the total size of its susceptibility set on $G^{(n)}$ exceeds $K$ is
less than $\varepsilon$ for all sufficiently large $n$; see Lemma \ref{lemignorehat}.
We then explore the forward process in $G^{(n)}$, where we ignore the
vertices and cliques already explored in the two backward processes. We
show that if the epidemic size is not $\Theta(1)$, then, with
probability tending to 1 as $n\to\infty$, it is $\Theta(n)$. After this
we attempt to connect the forward process with the generation $t_n$
vertices of the backward processes and show that, in the event of a
large outbreak, the probability that at least 1 of the vertices in
generation $t_n$ of a susceptibility set (if this generation is not
empty) is ultimately recovered converges to 1 as $n \to\infty$.

We use the following lemmas.
%
\begin{lemma}\label{Stirling}
Let $0 < \varepsilon< 3/\mathrm{e} -1$. For $k \in\mathbb{N}$,
let $(X_i(k),i\in\mathbb{N})$ be a sequence of i.i.d. $\Poisson
((1+\varepsilon)\log k)$ random variables. Then, for every $C>0$,
\[
\mathbb{P}\Bigl(\max_{1 \leq i \leq\lfloor C k \rfloor} X_i(k) \leq3 \log k
\Bigr) \to1 \qquad\mbox{as $k \to\infty$.}
\]
\end{lemma}
\begin{pf}
Since $\mathrm{e}^k = \sum_{i=0}^{\infty}k^i/i!$, we have $k!>k^k \mathrm{e}^{-k}$. Then
\begin{eqnarray*}
\mathbb{P}\bigl(X_1(k)> 3 \log k\bigr) & = & \sum
_{j=\lceil3\log k\rceil}^{\infty
} \frac{((1+\varepsilon)\log k)^j}{j!}\frac{1}{k^{1+\varepsilon}}
\\
& \leq& \frac{1}{k^{1+\varepsilon}}\sum_{j=\lceil3\log k\rceil}^{\infty
}
\frac{((1+\varepsilon)\log k)^j}{j^j \mathrm{e}^{-j}}
\\
& < & \frac{1}{k^{1+\varepsilon}}\sum_{j=\lceil3\log k\rceil}^{\infty}
\bigl((1+\varepsilon)\mathrm{e}/3\bigr)^{j}
\\
& < & \frac{3}{3-(1+\varepsilon)\mathrm{e}} k^{-1-\varepsilon+3(1+ \log
[1+\varepsilon]-\log3)}.
\end{eqnarray*}
The probability that none out of $\lfloor C k \rfloor$ independent
copies of $X_1(k)$ exceeds $3 \log k$ is thus given by
\begin{eqnarray*}
\bigl(1-\mathbb{P}\bigl(X_1(k)> 3 \log k\bigr)
\bigr)^{\lfloor Ck \rfloor} & > & \biggl(1-\frac{3}{3-(1+\varepsilon)\mathrm{e}} k^{-1-\varepsilon+3(1 + \log
[1+\varepsilon] -\log3)}
\biggr)^{C k}
\\
& > & 1 - C k \frac{3}{3-(1+\varepsilon)\mathrm{e}} k^{-1-\varepsilon+3(1+\log
[1+\varepsilon]-\log3 )}
\\
& = & 1- \frac{3C}{3-(1-\varepsilon)\mathrm{e}} k^{3(1 + \log[1+\varepsilon
]-\log3)-\varepsilon},
\end{eqnarray*}
which converges to 1 as $k \to\infty$, since $0 < \varepsilon< 3/\mathrm{e} -1$.
\end{pf}

Recall that the distance between two vertices in a graph is the number
of edges in the shortest path connecting those vertices.

\begin{lemma}\label{weightsus}
For $\nu$-almost all $\omega\in\Omega$, the probability that the
total number and the total weight of vertices within distance $2t_n$ of
the set $\{v_1,v_2\}$ in $\hat{\mathbb{A}}^{(n)}$ are both smaller than
$n^{1/3}$ converges to $1$ as $n\to\infty$.
\end{lemma}
\begin{pf}
All vertices in $\hat{\mathbb{A}}^{(n)}$ have weight at most $\log n$,
so their degrees in $\hat{\mathbb{A}}^{(n)}$ are stochastically
dominated by i.i.d. $\Poisson(\log n \max(\hat{L}^{(n)},\hat
{L}'^{(n)})/(\mu n))$ random variables. For every $\varepsilon>0$, we have
by the strong law of large numbers that $\ind(\max(\hat{L}^{(n)},\hat
{L}'^{(n)})/(\mu n) < 1+\varepsilon) \asarrows1$ as $n \to\infty$.
We know by Lemma \ref{Stirling} that, with probability tending to 1 as
$n\to\infty$, none of the at most $n + \lfloor\alpha n\rfloor$
vertices in $\hat{\mathbb{A}}^{(n)}$ has degree exceeding $3 \log n$.
Thus the number of vertices within graph distance $2t_n$ of $v_1$ and
$v_2$ is, with probability tending to 1 as $n\to\infty$, bounded above
by
\[
2 \sum_{k=1}^{2t_n} (3 \log
n)^{k} = O\bigl((3 \log n)^{2t_n+1}\bigr).
\]
Since $2t_n +1 = 2 \lceil\log\log n\rceil+1 < 2 \log\log n +3$, we have
\begin{eqnarray*}
(3 \log n)^{2t_n+1} & < & (3 \log n)^{3 + 2\log\log n}
\\
& = & (3 \log n)^3 \mathrm{e}^{2\log\log n(\log3 + \log\log n)} = o\bigl(n^{1/3}/
\log n\bigr),
\end{eqnarray*}
so the total weight of the vertices is $o(n^{1/3})$.
\end{pf}

For $i \in\{1,2\}$, let $K^i(t_n)$ be the set of vertices in $V^{(n)}$
within distance $2t_n$ of $v_i$ in $\hat{\mathbb{A}}^{(n)}$, and let
$K'^i(t_n)$ be the set of vertices in $V'^{(n)}$ within distance $2t_n$
of $v_i$ in $\hat{\mathbb{A}}^{(n)}$.
Lemma \ref{weightsus} implies that, with probability tending to 1 as
$n\to\infty$, none of the sets $K^1(t_n)$, $K^2(t_n)$, $K'^1(t_n)$ and
$K'^2(t_n)$ has total vertex or clique weight exceeding $n^{1/3}$.
Furthermore, with probability tending to 1 as $n\to\infty$, the total
number of vertices in $K^1(t_n)$ is less than $n^{1/3}$.
Conditioned on $K^2(t_n)$ having total weight less than $n^{1/3}$ and
$K^1(t_n)$ containing less than $n^{1/3}$ vertices, the probability
that $K^1(t_n)$ and $K^2(t_n)$ share an edge is bounded above by $1-
(1-n^{1/3}/\hat{L}_n)^{n^{1/3}} < n^{2/3}/\hat{L}_n$, which converges
$\nu$-almost surely to $0$ as $n \to\infty$. So, for $\nu$-almost all
$\omega\in\Omega$, the $\mathbb{P}_\omega$-probability that $K^1$ and
$K^2$ share a vertex converges to 0 as $n \to\infty$. Similarly, we
deduce that for $\nu$-almost all $\omega\in\Omega$, the $\mathbb
{P}_\omega$-probability that $K'^1$ and $K'^2$ share a clique converges
to 0 as $n \to\infty$.

Recall the definition of $R_*$ from \eqref{R*def} and write $R_*$ as
$R_*(A,B,\mathcal{I})$ to show explicitly its dependence on the
distributions of $A,B$ and $\mathcal{I}$.

\begin{lemma}\label{sizesus}
Suppose that $R_*>1$. Then, for $0< c < \log R_*$,
\[
\mathbb{P}_{\omega} \bigl(|\hat{\mathcal{S}}_{t_n}| > (\log
n)^{c} | |\hat{\mathcal{S}}_{t_n}| > 0 \bigr) \parrow1.
\]
\end{lemma}
\begin{pf}
First note that, since all vertices\vspace*{1pt} in $\hat{\mathbb{A}}^{(n)}$ have
weight $\le\log n$, the number of offspring of any individual in the
branching process $\mathcal{Z}^{b}(\hat{A}^{(n)},\break \hat{B}^{(n)},\mathcal
{I})$ is stochastically smaller than the product of two independent
$\Poisson(\log n)$ random variables.
Thus, a simple argument using Markov's inequality shows that the total
number of individuals in generations $0,1,\ldots,t_n$ of the branching
process $\mathcal{Z}^{b}(\hat{A}^{(n)},\hat{B}^{(n)},\mathcal{I})$ is
$O_{p_{\nu}}([(\log n)^2]^{\log\log n +1+\delta})$ for any $\delta>0$,
and hence $O_{p_{\nu}}(n^\varepsilon)$ for any $\varepsilon>0$.
Therefore, by choosing $\varepsilon<\frac{1}{3}$ (so that $2\varepsilon+\frac
{1}{3}<1$) and using Lemma~\ref{weightsus}, a standard coupling
argument, similar to that used in the proof of
Lemma \ref{fcoupling}, shows that with probability tending to $1$ as $n
\to\infty$, the susceptibility set process $\hat{\mathcal{S}}$ and the
branching process $\mathcal{Z}^{b}(\hat{A}^{(n)},\hat{B}^{(n)},\mathcal
{I})$ coincide over generations $0,1,\ldots,t_n$.
Thus, in proving Lemma \ref{sizesus}, we can replace $\hat{\mathcal
{S}}$ by $\mathcal{Z}^{b}(\hat{A}^{(n)},\hat{B}^{(n)},\mathcal{I})$.

For $n \in\mathbb{N}$, let $\hat{A}_*^{(n)}$ be a random variable
having distribution function given by
\[
\mathbb{P}_\omega\bigl(\hat{A}_*^{(n)} \le x\bigr) = \sup
_{i \geq n} \mathbb {P}_\omega\bigl(\hat{A}^{(i)}
\le x\bigr)\qquad (x \in\mathbb{R})
\]
and define $\hat{B}_*^{(n)}$ similarly.
Observe that $\hat{A}_*^{(n)} \Rightarrow A$ and $\hat{B}_*^{(n)}
\Rightarrow B$ as $n \to\infty$. Furthermore, for all $n \in\mathbb
{N}$, $\hat{A}_*^{(n)}$ (resp., $\hat{B}_*^{(n)}$) is stochastically
dominated by $\hat{A}_*^{(n+1)}$ (resp., $\hat{B}_*^{(n+1)}$).
Therefore $R_*(\hat{A}_*^{(n)},\hat{B}_*^{(n)},\mathcal{I})$ is also
stochastically increasing in $n$.
By the Skorokhod representation theorem (Theorem 7.2.14 of \cite{Grim92})
and the monotone convergence theorem we have that
\[
R_*\bigl(\hat{A}_*^{(n)},\hat{B}_*^{(n)},\mathcal{I}\bigr)
\parrow R_*(A,B,\mathcal{I}).
\]
In particular, there exists $N=N(\omega)$ such that
$R_*(\hat{A}_*^{(n)},\hat{B}_*^{(n)},\mathcal{I}) > \mathrm{e}^c$, for
every $n>N$.
So, by Theorem 2.7.1 of \cite{Jage75}, it follows that
\[
\mathbb{P}_{\omega}\bigl(\bigl|\mathcal{Z}_{t_n}^{b}
\bigl(\hat{A}_*^{(n)},\hat {B}_*^{(n)},\mathcal{I}\bigr)\bigr| > (
\log n)^{c}\bigr)- \mathbb{P}_{\omega}\bigl(\bigl|
\mathcal{Z}_{t_n}^{b}\bigl(\hat{A}_*^{(n)},\hat
{B}_*^{(n)},\mathcal{I}\bigr)\bigr| > 0\bigr) \parrow0.
\]
The second probability in this expression converges to $\rho
^b(A,B,\mathcal{I})$ by Lem\-ma~4.1 of \cite{Brit07} and the lemma then
follows by observing that $|\mathcal{Z}_{t_n}^{b}(\hat{A}_*^{(n)},\break \hat
{B}_*^{(n)},\mathcal{I})|$ is stochastically smaller than $|\mathcal
{Z}_{t_n}^{b}(\hat{A}^{(n)},\hat{B}^{(n)},\mathcal{I})|$.
\end{pf}

Up to now, we have investigated the behavior of the susceptibility sets
of vertices in $\hat{G}^{(n)}$. This is only an intermediate step
before analyzing susceptibility sets in $G^{(n)}$. To make the
connection between the two graphs, we use the following two lemmas.
%
\begin{lemma}\label{hatdoesntmatter}
For $k \in\mathbb{N}$,
\[
\mathbb{P}_{\omega}\bigl(\bigl|\hat{\mathcal{S}}\bigl(\hat{A}^{(n)},
\hat {B}^{(n)},\mathcal{I}\bigr)\bigr|=k\bigr) - \mathbb{P}_{\omega}
\bigl(\bigl|\mathcal {S}\bigl(A^{(n)},B^{(n)},\mathcal{I}\bigr)\bigr|=k
\bigr) \parrow0.
\]
\end{lemma}
\begin{pf}
In order to simplify the notation we suppress the explicit dependence
on $\hat{A}^{(n)}$, $\hat{B}^{(n)}$ and $\mathcal{I}$. We denote by
$\mathcal{S}'^i$ the set of cliques containing vertices in the
susceptibility set $\mathcal{S}^i$. We prove that,
for all $k,l\in\mathbb{N}$,
%
%
\begin{equation}
\mathbb{P}_{\omega}\bigl(|\hat{\mathcal{S}}|=k, \bigl|\hat{
\mathcal{S}}'\bigr|=l\bigr) - \mathbb{P}_{\omega}\bigl(|
\mathcal{S}|=k,\bigl|\mathcal{S}'\bigr|=l\bigr) \parrow0, \label{SS'sameProbs}
\end{equation}
from which the lemma follows using similar arguments to those in the
proof of Lemma \ref{bplimits}, which are not repeated here.

Recall that we can construct $G^{(n)}$ from $\hat{G}^{(n)}$, by
considering the vertices in $V^{(n)}\setminus\hat{V}^{(n)}$ and
$V'^{(n)}\setminus\hat{V}'^{(n)}$ and then connecting them in the
usual way with each other and with vertices in $V^{(n)}$ and $V'^{(n)}$
to obtain $\mathbb{A}^{(n)}$. As in the proof of Lemma~\ref{fcoupling},
$\mu< \infty$ implies that
\[
\sum_{i=1}^n A_i \ind
\bigl(A_i > \log \lfloor n\rfloor \bigr) = L^{(n)} - \hat{L}^{(n)}
= o(n)\qquad \mbox{$\nu$-almost surely}.
\]
Therefore,
\[
\frac{L^{(n)} - \hat{L}^{(n)}}{L^{(n)}} \asarrows0\qquad \mbox{as $n \to\infty$}.
\]
This implies that $1 - \hat{L}^{(n)}/L^{(n)}$ converges in probability
to 0. In particular there is an increasing sequence of natural numbers
$(p_i, i \in\mathbb{N})$, such that for all $n > p_i$, we have $\nu(1
- \hat{L}^{(n)}/L^{(n)} < 4^{-i})>1-2^{-i}$. Define the function $\xi\dvtx
\mathbb{N} \to\mathbb{N}$ by $\xi(n) = 2^i$ if $ p_i \leq n <
p_{i+1}$. This function increases to infinity and
\[
\ind \bigl(L^{(n)} - \hat{L}^{(n)}< \bigl(\xi(n)
\bigr)^{-1}L^{(n)} \bigr) \parrow1. 
\]
Similarly, there exists a function $\xi'(n)$ which increases to $\infty
$, such that
%
%
\begin{equation}
\ind \bigl(L'^{(n)} - \hat{L}'^{(n)}<
\bigl(\xi'(n)\bigr)^{-1}L'^{(n)}
\bigr) \parrow1. \label{L'L'hat}
\end{equation}

Let $\hat{L}^{(n)}_{(k)}$ (resp., $\hat{L}'^{(n)}_{(k)}$) be the weight
of the first $k$ vertices from $\hat{V}^{(n)}$ (resp., $\hat
{V}'^{(n)}$) explored in $\hat{\mathcal{S}}$. Note that
\[
\mathbb{P}_{\omega} \Bigl(|\hat{\mathcal{S}}|=k, \bigl|\hat{
\mathcal{S}}'\bigr|=l \Bigm| \hat{L}^{(n)}_{(k)}\geq\bigl(
\xi'(n)\bigr)^{1/2} \cup\hat {L}'^{(n)}_{(l)}
\geq\bigl(\xi(n)\bigr)^{1/2} \Bigr) \parrow0,
\]
since if the conditioning event occurs, then the probability that the
susceptibility set does not extend further goes to 0 as $n \to\infty$.
It follows that
%
%
\begin{eqnarray}
\label{susetext}&& \mathbb{P}_{\omega} \bigl(|\hat{\mathcal{S}}|=k, \bigl|\hat{
\mathcal{S}}'\bigr|=l, \hat{L}^{(n)}_{(k)}< \bigl(
\xi'(n)\bigr)^{1/2}, \hat{L}'^{(n)}_{(l)}<
\bigl(\xi (n)\bigr)^{1/2} \bigr)\nonumber\\
 &&\quad{}- \mathbb{P}_{\omega}\bigl(|
\hat{\mathcal{S}}|=k, \bigl|\hat{\mathcal{S}}'\bigr|=l\bigr)\\
&&\qquad\parrow0.\nonumber
\end{eqnarray}
Given $\omega$, when constructing the graph $G^{(n)}$ from $\hat
{G}^{(n)}$, the expected number of newly-added edges between the first
$k$ vertices from $\hat{V}^{(n)}$ explored in $\hat{\mathcal{S}}$ and
$V^{'(n)}\setminus\hat{V}^{'(n)}$ is
\[
F^{(n)}_k=\frac{\hat{L}^{(n)}_{(k)}(L'^{(n)} - \hat{L}'^{(n)})}{\mu n}.
\]
Suppose that $\hat{L}^{(n)}_{(k)}< (\xi'(n))^{1/2}$. Then
\[
F^{(n)}_k \le\bigl(\xi'(n)
\bigr)^{1/2}\frac{(L'^{(n)} - \hat
{L}'^{(n)})}{L'^{(n)}}\frac{L'^{(n)}}{n\mu},
\]
which, together with \eqref{L'L'hat} and the fact that $L'^{(n)}/(n \mu
) \asarrows1$ as $n\to\infty$, yields
\[
F^{(n)}_k \ind\bigl(\hat{L}^{(n)}_{(k)}<
\bigl(\xi'(n)\bigr)^{1/2}\bigr)\parrow0.
\]
Combining this, and a corresponding result for the number of
newly-added edges between the first $l$ vertices from $\hat{V}^{'(n)}$
explored in $\hat{\mathcal{S}}$ and $V^{(n)}\setminus\hat{V}^{(n)}$,
with \eqref{susetext}
establishes that
\[
\mathbb{P}_{\omega} \bigl(|\hat{\mathcal{S}}|=k,\bigl |\hat{
\mathcal{S}}'\bigr|=l, \mathcal{S} \cap\bigl(V^{(n)} \setminus
\hat{V}^{(n)}\bigr) \neq\varnothing, \mathcal{S}' \cap
\bigl(V'^{(n)} \setminus\hat{V}'^{(n)}
\bigr) \neq\varnothing \bigr)
\parrow 0,
\]
which completes the proof of \eqref{SS'sameProbs} and thus of the lemma.
\end{pf}

\begin{lemma}\label{lemignorehat}
For every $\varepsilon>0$ there exists $K\in\mathbb{N}$ such that
\[
\ind\bigl(\mathbb{P}_{\omega}\bigl(\bigl|\hat{\mathcal{S}}_{t_n}
\bigl(\hat{A}^{(n)},\hat {B}^{(n)},\mathcal{I}\bigr)\bigr|=0,\bigl|
\mathcal{S}\bigl(A^{(n)},B^{(n)},\mathcal{I}\bigr)\bigr|>K\bigr) <
\varepsilon\bigr) \parrow1.
\]
\end{lemma}
\begin{pf}
For ease of presentation we suppress the dependence on the
distributions of the weights and infectious periods, writing $\hat
{\mathcal{S}}$ for $\hat{\mathcal{S}}(\hat{A}^{(n)},\hat
{B}^{(n)},\mathcal{I})$ and $\mathcal{S}$ for $\mathcal
{S}(A^{(n)},B^{(n)},\mathcal{I})$. First note that, as in the proof of
Lemma \ref{sizesus}, we can use branching process approximations to
show that for every $K\in\mathbb{N}$ we have
%
%
\begin{eqnarray}
\label{lemignorehat1} &&\mathbb{P}_{\omega}\bigl(|\hat{\mathcal{S}}_{t_n}|=0,|
\hat{\mathcal{S}}|>K\bigr) \nonumber\\
&&\quad{}- \mathbb{P}_{\omega}\bigl(\bigl|\mathcal{Z}^b_{t_n}
\bigl(\hat{A}^{(n)},\hat {B}^{(n)},\mathcal{I}\bigr)\bigr|=0,\bigl|
\mathcal{Z}^b\bigl(\hat{A}^{(n)},\hat {B}^{(n)},
\mathcal{I}\bigr)\bigr|>K\bigr)\\
&&\qquad\parrow0.\nonumber
\end{eqnarray}
Now,
%
%
\begin{eqnarray}\label{lemignorehat2}
&&\mathbb{P}_{\omega} \bigl(\bigl|\mathcal{Z}^b_{t_n}
\bigl(\hat{A}^{(n)},\hat {B}^{(n)},\mathcal{I}\bigr)\bigr|=0,\bigl|
\mathcal{Z}^b\bigl(\hat{A}^{(n)},\hat {B}^{(n)},
\mathcal{I}\bigr)\bigr|>K \bigr)
\nonumber
\\
&&\qquad = \mathbb{P}_{\omega} \bigl(\bigl|\mathcal{Z}^b\bigl(
\hat{A}^{(n)},\hat {B}^{(n)},\mathcal{I}\bigr)\bigr|>K \bigr)
\nonumber
\\[-8pt]
\\[-8pt]
\nonumber
&&\quad\qquad{} -\mathbb{P}_{\omega} \bigl(\bigl|\mathcal{Z}^b_{t_n}
\bigl(\hat {A}^{(n)},\hat{B}^{(n)},\mathcal{I}\bigr)\bigr|>0,\bigl|
\mathcal{Z}^b\bigl(\hat {A}^{(n)},\hat{B}^{(n)},
\mathcal{I}\bigr)\bigr|>K \bigr)
\\
&&\qquad = \mathbb{P}_{\omega} \bigl(\bigl|\mathcal{Z}^b\bigl(
\hat{A}^{(n)},\hat {B}^{(n)},\mathcal{I}\bigr)\bigr|>K\bigr) -
\mathbb{P}_{\omega}\bigl(\bigl|\mathcal {Z}^b_{t_n}\bigl(
\hat{A}^{(n)},\hat{B}^{(n)},\mathcal{I}\bigr)\bigr|>0
\bigr)\nonumber
\end{eqnarray}
for all sufficiently large $n$, since $|\mathcal{Z}^b_{t_n}(\hat
{A}^{(n)},\hat{B}^{(n)},\mathcal{I})|>0$ implies that $|\mathcal
{Z}^b(A^{(n)},\allowbreak B^{(n)},\mathcal{I})|>t_n$.

Arguing as in the proof of Lemma \ref{bplimits} shows that
%
%
\begin{equation}
\mathbb{P}_{\omega}\bigl(\bigl|\mathcal{Z}^b\bigl(
\hat{A}^{(n)},\hat{B}^{(n)},\mathcal {I}\bigr)\bigr|>K\bigr) \parrow
\mathbb{P}_{\omega}\bigl(\bigl|\mathcal{Z}^b(A,B,\mathcal{I})\bigr|>K
\bigr). \label{lemignorehat3}
\end{equation}
To deal with the second term on the right-hand side of \eqref
{lemignorehat2}, observe that
%
%
\begin{eqnarray}
\label{lemignorehat5}&& \mathbb{P}_{\omega} \bigl(\bigl|\mathcal{Z}^b_{t_n}
\bigl(\hat{A}^{(n)},\hat {B}^{(n)},\mathcal{I}\bigr)\bigr|>0\bigr) \nonumber\\
&&\qquad
= \mathbb{P}_{\omega}\bigl(\bigl|\mathcal{Z}^b\bigl(
\hat{A}^{(n)},\hat {B}^{(n)},\mathcal{I}\bigr)\bigr|=\infty\bigr)
\\
& &\qquad\quad{} + \mathbb{P}_{\omega}\bigl(\bigl|\mathcal{Z}^b_{t_n}
\bigl(\hat{A}^{(n)},\hat {B}^{(n)},\mathcal{I}\bigr)\bigr|>0, \bigl|
\mathcal{Z}^b\bigl(\hat{A}^{(n)},\hat {B}^{(n)},
\mathcal{I}\bigr)\bigr|<\infty\bigr)\nonumber
\end{eqnarray}
and
\begin{eqnarray*}
&&\mathbb{P}_{\omega}\bigl(\bigl|\mathcal{Z}^b_{t_n}\bigl(
\hat{A}^{(n)},\hat {B}^{(n)},\mathcal{I}\bigr)\bigr|>0, \bigl|
\mathcal{Z}^b\bigl(\hat{A}^{(n)},\hat {B}^{(n)},
\mathcal{I}\bigr)\bigr|<\infty\bigr)
\\
&&\qquad\le \mathbb{P}_{\omega}\bigl(\bigl|
\mathcal{Z}^b\bigl(\hat{A}^{(n)},\hat{B}^{(n)},
\mathcal {I}\bigr)\bigr| \in(t_n,\infty)\bigr).
\end{eqnarray*}

Now, given any $\varepsilon>0$, there exists $L\in\mathbb{N}$ such that $
\mathbb{P}(|\mathcal{Z}^b(A,B,\mathcal{I})| \in(L,\infty))<\varepsilon$.
(If $R_* \le1$, then $|\mathcal{Z}^b|$ is almost surely finite and the
statement follows immediately. If $R_* > 1$, the statement follows by
writing $ \mathbb{P}(|\mathcal{Z}^b| \in(L,\infty))=\rho^b\mathbb
{P}(|\mathcal{Z}^b| \in(L,\infty)\mid |\mathcal{Z}^b|<\infty)$ and using
the fact that a supercritical Galton--Watson process conditioned on
extinction is probabilistically equivalent to an associated subcritical
Galton--Watson process \cite{Daly79}.) Further, \eqref{lemignorehat3}
and Lemma~4.1 of \cite{Brit07}, imply that
\[
\mathbb{P}_{\omega}\bigl(\bigl|\mathcal{Z}^b\bigl(
\hat{A}^{(n)},\hat{B}^{(n)},\mathcal {I}\bigr)\bigr| \in(L,\infty)
\bigr) \parrow \mathbb{P}\bigl(\bigl|\mathcal{Z}^b(A,B,\mathcal{I})\bigr| \in(L,
\infty)\bigr),
\]
so
\[
\ind\bigl(\mathbb{P}_{\omega}\bigl(\bigl|\mathcal{Z}^b\bigl(
\hat{A}^{(n)},\hat {B}^{(n)},\mathcal{I}\bigr)\bigr| \in(L,\infty)
\bigr) < \varepsilon\bigr) \parrow1,
\]
which implies that
\[
\ind\bigl(\mathbb{P}_{\omega}\bigl(\bigl|\mathcal{Z}^b\bigl(
\hat{A}^{(n)},\hat {B}^{(n)},\mathcal{I}\bigr)\bigr|
\in(t_n,\infty)\bigr) < \varepsilon\bigr) \parrow1.
\]
As this holds for any $\varepsilon>0$, it follows from \eqref
{lemignorehat2}, \eqref{lemignorehat3} and \eqref{lemignorehat5}, with
another application of Lemma 4.1 of \cite{Brit07}, that
%
%
\begin{eqnarray}
\label{lemignorehat6} && \mathbb{P}_{\omega}\bigl(\bigl|\mathcal{Z}_{t_n}^b
\bigl(\hat{A}^{(n)},\hat {B}^{(n)},\mathcal{I}\bigr)\bigr| =0, \bigl|
\mathcal{Z}^b\bigl(\hat{A}^{(n)},\hat {B}^{(n)},
\mathcal{I}\bigr)\bigr|>K\bigr)
\nonumber
\\[-8pt]
\\[-8pt]
\nonumber
&&\qquad \parrow\mathbb{P}\bigl(\bigl|\mathcal{Z}^b(A,B,\mathcal{I})\bigr| \in (K,
\infty)\bigr).
\end{eqnarray}

Now $\mathbb{P}(|\mathcal{Z}^b(A,B,\mathcal{I})| \in(K,\infty))$ can be
made arbitrarily close to 0 by choosing~$K$ sufficiently large.
Thus \eqref{lemignorehat1} and \eqref{lemignorehat6} imply that, for
every $\varepsilon>0$, we can choose $K\in\mathbb{N}$ such that
%
%
\begin{equation}
\ind\bigl( \mathbb{P}_{\omega}\bigl(|\hat{\mathcal{S}}_{t_n}|=0,|
\hat{\mathcal {S}}|>K\bigr) < \varepsilon\bigr) \parrow1. \label{lemignorehat7}
\end{equation}

Finally, note that
\begin{eqnarray*}
\mathbb{P}_{\omega}\bigl(|\hat{\mathcal{S}}_{t_n}|=0,|\hat{
\mathcal{S}}|>K\bigr) & = & \mathbb{P}_{\omega}\bigl(|\hat{\mathcal{S}}_{t_n}|=0\bigr)
- \mathbb {P}_{\omega}\bigl(|\hat{\mathcal{S}}_{t_n}|=0,|\hat{
\mathcal{S}}| \leq K\bigr)
\\
& = & \mathbb{P}_{\omega}\bigl(|\hat{\mathcal{S}}_{t_n}|=0\bigr) -
\mathbb {P}_{\omega}\bigl(|\hat{\mathcal{S}}| \leq K\bigr)
\end{eqnarray*}
for all sufficiently large $n$. Similarly, since $|\mathcal{S}| \geq
|\hat{\mathcal{S}}|$,
\[
\mathbb{P}_{\omega}\bigl(|\hat{\mathcal{S}}_{t_n}|=0, |
\mathcal{S}|>K\bigr) = \mathbb{P}_{\omega}\bigl(|\hat{\mathcal{S}}_{t_n}|=0\bigr)
- \mathbb{P}_{\omega
}(|\mathcal{S}| \leq K)
\]
for all sufficiently large $n$. Hence, by Lemma \ref{hatdoesntmatter},
\[
\mathbb{P}_{\omega}\bigl(|\hat{\mathcal{S}}_{t_n}|=0, | \hat{
\mathcal{S}}|>K\bigr) - \mathbb{P}_{\omega}\bigl(|\hat{\mathcal{S}}_{t_n}|=0,
| \mathcal{S}|>K\bigr) \parrow0,
\]
whence the lemma follows from \eqref{lemignorehat7}.
\end{pf}

For the remainder of the proof of Theorem \ref{finalsizethm}, we
re-analyze an exploration process of the forward epidemic process, and
we couple it to a multi-type branching process, such that the epidemic
process is bigger than the branching process for as long as the total
weight of both the vertices and the clique vertices in the exploration
process is less than a predefined fraction of the total weight. The
survival probability of this branching process can be made arbitrarily
close to the probability of a large outbreak as $n \to\infty$. After
that we ``glue'' the susceptibility sets, if they are large, to the
forward epidemic process.

We need some extra notation.
Since the weights of the vertices are exchangeable, the model does not
change if we order the vertices such that $A_i^{(n)} \leq
A_{i+1}^{(n)}$, and $B_j^{(n)} \leq B_{j+1}^{(n)}$, for $1\leq i <n$
and $1 \leq j < \lfloor\alpha n \rfloor$.
For $\gamma\in(0,1)$, we define
\begin{eqnarray*}
R^{(n)}(\gamma) & = &\min \biggl(i \leq n\dvtx \frac{\sum_{j=1}^i
A_j}{L^{(n)}} \geq1-
\gamma \biggr)\quad \mbox{and}
\\
R'^{(n)}(\gamma) & =& \min \biggl(i \leq\lfloor\alpha n
\rfloor\dvtx \frac
{\sum_{j=1}^i B_j}{L'^{(n)}} \geq1-\gamma \biggr).
\end{eqnarray*}
Furthermore, define
\begin{eqnarray*}
\bar{\gamma}&=& \bar{\gamma}(\gamma,n)  = 1- \frac{\sum_{j=1}^{R^{(n)}(\gamma)} A_j}{L^{(n)}}\quad \mbox{and}
\\
\bar{\gamma}'&=& \bar{\gamma}'(\gamma,n)  = 1-
\frac{\sum_{j=1}^{R'^{(n)}(\gamma)} B_j}{L'^{(n)}}.
\end{eqnarray*}
We claim that, for $\gamma\in(0, 1)$, $\bar{\gamma} \parrow\gamma$.
This can be seen by the following reasoning.
Let $x = \inf(y\geq0\dvtx \mu^{-1}\mathbb{E}[A \ind(A < y)]> 1-\gamma/2)$.
Then $x$ is finite, since $\mu= \mathbb{E}[A]<\infty$.
By the strong law of large numbers, we have
$n^{-1}\sum_{i=1}^n A_i \ind(A_i \leq x) \asarrows\mathbb{E}[A \ind(A
\leq x)]$ and $n^{-1} L^{(n)} \asarrows\mu$ as $n\to\infty$. Thus
\[
\frac{\sum_{i=1}^n A_i \ind(A_i\leq x)}{L^{(n)}} \asarrows\mu ^{-1}\mathbb{E}\bigl[A \ind(A \leq x)
\bigr] \geq1 - \gamma/2
\]
as $n \to\infty$, whence $\nu(A_{R^{(n)}} \leq x) \to1$ as $n \to
\infty$.
Combining this with
\[
1-\bar{\gamma} = \frac{\sum_{j=1}^{R^{(n)}(\gamma)} A_j}{L^{(n)}} \geq 1-\gamma
\]
and
\[
1-\bar{\gamma} - \frac{A_{R^{(n)}}}{L^{(n)}} = \frac{\sum_{j=1}^{R^{(n)}(\gamma)-1} A_j}{L^{(n)}} < 1-\gamma
\]
completes the proof of the claim.
Similarly we can prove that $\bar{\gamma}' \parrow\gamma$. This also
shows that the vertices in $V^{(n)} \setminus\hat{V}^{(n)}$ [resp.,
$V'^{(n)} \setminus\hat{V}'^{(n)}$] all have labels exceeding
$R^{(n)}(\gamma)$ [resp., $R'^{(n)}(\gamma)$] with probability tending
to 1 as $n\to\infty$.

For $c_1>0$, let $I(c_1)$ be the set of vertices with type/infectious
period less than $c_1$. Let $\mathcal{I}(c_1)$ denote a random variable
having distribution function given by $\mathbb{P}(\mathcal{I}(c_1) \leq
x) = \mathbb{P}(\mathcal{I} \leq x\mid \mathcal{I}\geq c_1)$, for $x \geq
c_1$. We use the multi-type branching process $\mathcal
{Z}^{f}(A^{(n)},B^{(n)},\mathcal{I}(c_1), \gamma)$, which is obtained
from $\mathcal{Z}^{f}(A^{(n)},B^{(n)},\mathcal{I}(c_1))$ by:
\begin{longlist}[(ii)]
\item[(i)] Killing upon birth all children with $A$-weight strictly
larger than the weight of vertex $R^{(n)}(\gamma)$. Children with
$A$-weight equal to the weight of vertex $R^{(n)}(\gamma)$ are killed
independently with probability given by the fraction of those vertices
in $V^{(n)}$ having weight equal to the weight of vertex $R^{(n)}(\gamma
)$ that also have label strictly larger than $R^{(n)}(\gamma)$.
\item[(ii)] Killing upon birth all litters corresponding to local
epidemics in cliques with $B$-weight strictly larger than the weight of
vertex $R'^{(n)}(\gamma)$. Cliques with $B$-weight equal to the weight
of clique $R'^{(n)}(\gamma)$ are killed independently with probability
given by the fraction of those vertices in $V'^{(n)}$ having $B$-weight
equal to the weight of clique $R'^{(n)}(\gamma)$ that also have label
strictly larger than $R'^{(n)}(\gamma)$.
\end{longlist}
If $A_1,A_2,\ldots,A_n$ are distinct, which happens $\nu$-almost surely
if the distribution of $A$ has no atoms, then (i) reduces to killing
upon birth all children with $A$-weight strictly larger than the weight
of vertex $R^{(n)}(\gamma)$. If $B_1,B_2,\ldots,B_{\lfloor\alpha
n\rfloor}$ are distinct, then (ii) simplifies similarly.

We observe that the corresponding survival probability function (cf. Section~\ref{forwardBP})
$\tilde{\rho}(x;A^{(n)},B^{(n)},\mathcal{I}(c_1),\gamma)$ increases as
$\gamma\downarrow0$. Thus, the limit function,
as $\gamma\downarrow0$,
exists and satisfies \eqref{survivetilde} by the monotone convergence
theorem. Invoking Lemma~\ref{unilem}, this limit function is
\[
\lim_{\gamma\downarrow0} \tilde{\rho}\bigl(x;A^{(n)},B^{(n)},
\mathcal {I}(c_1), \gamma\bigr) = \tilde{\rho}\bigl(x;A^{(n)},B^{(n)},
\mathcal{I}(c_1)\bigr).
\]
Similarly, since $\tilde{\rho}(x;A^{(n)},B^{(n)},\mathcal{I}(c_1))$ is
decreasing as $c_1 \downarrow0$, one can show that
\[
\lim_{c_1 \downarrow0} \tilde{\rho}\bigl(x;A^{(n)},B^{(n)},
\mathcal {I}(c_1)\bigr) = \tilde{\rho}\bigl(x;A^{(n)},B^{(n)},
\mathcal{I}\bigr).
\]
For $\rho(A^{(n)},B^{(n)},\mathcal{I})$ as in Section~\ref{forwardBP},
this leads to the first assertion of the following lemma. The second
assertion then follows using Lemma \ref{survfin}.
%
\begin{lemma}\label{killprocess}
For every $\varepsilon> 0$, $\omega\in\Omega$ and $n \in\mathbb{N}$,
there exist $\gamma>0$ and $c_1>0$ small enough such that
\[
\bigl|\rho\bigl(A^{(n)},B^{(n)},\mathcal{I}(c_1),
\gamma\bigr) - \rho \bigl(A^{(n)},B^{(n)},\mathcal{I}\bigr)\bigr| <
\varepsilon/2.
\]
For every $\varepsilon>0$, there exist $\gamma>0$ and $c_1>0$ such that
\[
\ind\bigl(\bigl|\rho\bigl(A^{(n)},B^{(n)},\mathcal{I}(c_1),
\gamma\bigr) -\rho(A,B,\mathcal {I})\bigr| < \varepsilon\bigr) \parrow1.
\]
\end{lemma}

Let $c_1>0$ and $\gamma\geq0$ be constants. We consider the forward
epidemic process $\bar{\mathcal{R}}^{(n,\gamma)} = \bar{\mathcal
{R}}^{(n)}(\omega,\mathcal{I}, c_1,\gamma/3)$, which is obtained from
$\mathcal{R}^{(n)}(\omega,\mathcal{I})$ by removing all vertices (and
adjacent edges) in $I(c_1)$, $K^1(t_n)$ and $K^2(t_n)$ and not allowing
for contacts in the cliques $K'^1(t_n)$ and $K'^2(t_n)$ or in cliques
with label $R'^{(n)}(\gamma/3)$ or larger. As before, we deduce that
for every $\gamma>0$ and large enough $n$, all vertices in
$V'^{(n)}\setminus\hat{V}'^{(n)}$ have label at least $R'^{(n)}(\gamma
/3)$, with probability arbitrarily close to 1.
Also define $\bar{\mathcal{R}}^{(n)} = \bar{\mathcal{R}}^{(n,0)} = \bar
{\mathcal{R}}(\omega,\mathcal{I}, c_1,0)$, and let the total weight of
the cliques in $\bar{\mathcal{R}}^{(n)}$ (i.e., in the set of
ultimately recovered vertices in $\bar{\mathcal{R}}^{(n)}$) be denoted
by $\bar{\mathcal{W}}'^{(n)}(c_1)$.

\begin{lemma}\label{eplarbra}
Suppose that $R_*>1$. Then for every $\varepsilon>0$, there exist
constants $\eta>0$ and $c_1>0$, such that
\[
\ind \bigl(\mathbb{P}_{\omega}\bigl(\bar{\mathcal{W}}'^{(n)}(c_1)>
\eta n\bigr) - \bigl(\rho(A,B,\mathcal{I}) -\varepsilon\bigr)>0 \bigr) \parrow1.
\]
\end{lemma}
\begin{pf}
We explore $\bar{\mathcal{R}}^{(n,\gamma)}$ vertex by vertex (and
clique by clique) and couple this with an exploration process of the
tree of the branching process
\[
\mathcal{Z}^{(n,\gamma)} = \mathcal{Z}^{f}\bigl(
\hat{A}^{(n)},\hat {B}^{(n)},\mathcal{I}(c_1),
\gamma\bigr).
\]
With some abuse of notation we use $\bar{\mathcal{R}}^{(n,\gamma)}$ and
$\mathcal{Z}^{(n,\gamma)}$ for the exploration processes as well.

We choose one vertex uniformly at random from $\hat{V}^{(n)}$. We
assume that this vertex is not in $K^1(t_n)$ or $K^2(t_n)$ and that its
type/infectious period exceeds $c_1$. The probability that this
assumption is met can be made arbitrarily close to 1 by choosing $n$
large enough and $c_1$ small enough. Denote this vertex by $\bar{v}_0$.
Define the ``forbidden sets'' of vertices by
\begin{eqnarray*}
\Gamma_0 & = &K^1(t_n) \cup
K^2(t_n) \cup I(c_1) \cup
\bigl(V^{(n)} \setminus \hat{V}^{(n)}\bigr) \cup\{
\bar{v}_0\} \quad\mbox{and}
\\
\Gamma'_0 & =& K'^1(t_n)
\cup K'^2(t_n) \cup\bigl\{
v'_i \in V'^{(n)}\dvtx i \geq
R'^{(n)}(\gamma/3)\bigr\}.
\end{eqnarray*}
For the vertices in $V^{(n)} \setminus\Gamma_0$, we re-randomize the
infectious period in such a way that, for every vertex in $V^{(n)}
\setminus\Gamma_0$, we let it be an independent random variable with
distribution $\mathcal{I}(c_1)$. This will not affect the distribution
of the processes.

Let $\sigma^{(n)}_0(i)$ be a relabeling of the vertices in $V^{(n)}$
such that if $v_j \in\Gamma_0$ and $v_i \in V^{(n)} \setminus\Gamma
_0$, then $\sigma^{(n)}_0(i) < \sigma^{(n)}_0(j)$, while if $v_i,v_j
\in V^{(n)} \setminus\Gamma_0$, then $\sigma^{(n)}_0(i) < \sigma
^{(n)}_0(j)$ if $i < j$. The precise order of the labels of the
vertices in the forbidden set is not important.
Define $\sigma'^{(n)}_0(i)$ similarly.

The $A$-weight and type of $\bar{v}_0$ are also assigned to the
ancestor of $\mathcal{Z}^{(n,\gamma)}$, say that the $A$-weight is
$a_0$. Then we use a $\Poisson(a_0 L'^{(n)}/(\mu n))$ random variable,
$d_0$, to denote the ``maximal'' number of cliques vertex $\bar{v}_0$ is
part of and, coupled to this, the ``maximal'' number of child cliques the
vertex has in $\mathcal{Z}^{(n,\gamma)}$. The meaning of maximal is
clarified below.

We now identify the first child clique. Choose a real number, $x'$ say,
uniformly at random from the unit interval.
In $\bar{\mathcal{R}}^{(n,\gamma)}$ we try to connect vertex $\bar
{v}_0$ to the clique with label $i$, which satisfies
\[
\sum_{j \in\mathbb{N}\dvtx \sigma'^{(n)}_0(j) < \sigma'^{(n)}_0(i)} B_{j} < x'L'^{(n)}
\leq\sum_{j \in\mathbb{N}\dvtx \sigma'^{(n)}_0(j) \leq\sigma
'^{(n)}_0(i)} B_{j}.
\]
Let this vertex be $\bar{v}'_1$.
The $B$-weight of the corresponding possible litter in $\mathcal
{Z}^{(n,\gamma)}$ is $B_i$, where $i$ is such that
$\sum_{j = 1}^{i-1} B_k < x' L'^{(n)} \leq\sum_{j = 1}^{i} B_j$.
If $\bar{v}'_1 \in\Gamma'_0$, then the clique is ignored in $\bar
{\mathcal{R}}^{(n,\gamma)}$.
If $x' > 1- \bar{\gamma}'$, then the litter in $\mathcal{Z}^{(n,\gamma
)}$ is ignored.
We note that as long as the weight of $\Gamma'_0$ is less than $\bar
{\gamma}L'^{(n)}$, a clique can be ignored in $\bar{\mathcal
{R}}^{(n,\gamma)}$ only if the corresponding litter in $\mathcal
{Z}^{(n,\gamma)}$ is also ignored. Furthermore, the $B$-weight of the
litter in $\mathcal{Z}^{(n,\gamma)}$ is not larger than the $B$-weight
of the clique in $\bar{\mathcal{R}}^{(n,\gamma)}$.

Let the label of $\bar{v}'_1$ be $k$.
We now define
\[
\sigma'^{(n)}_1(i) = \cases{
\sigma'^{(n)}_0(i), &\quad  $\mbox{for $i$ such
that $\sigma'^{(n)}_0(i)<
\sigma'^{(n)}_0(k)$,}$ \vspace*{2pt}
\cr
\sigma'^{(n)}_0(i) -1, & \quad $\mbox{for $i$ such
that $\sigma'^{(n)}_0(i)>
\sigma'^{(n)}_0(k)$,}$ \vspace*{2pt}
\cr
\lfloor\alpha n \rfloor, & \quad $\mbox{for $i=k$.}$}
\]
That is, we give $\bar{v}'_1$ the maximal label and keep the order of
the labels of the other vertices. Furthermore, we add $\bar{v}'_1$ to
the forbidden set, that is, set $\Gamma'_1 = \Gamma'_0 \cup\{\bar
{v}'_1\}$.
We choose the next clique in $\bar{\mathcal{R}}^{(n,\gamma)}$ and
corresponding litter in $\mathcal{Z}^{(n,\gamma)}$, say $\bar{v}'_2$,
in the same way as we choose $\bar{v}'_1$, with $\sigma'^{(n)}_0$
replaced by $\sigma'^{(n)}_1$ and $\Gamma'_0$ replaced by $\Gamma'_1$,
and we continue this process until we have identified all cliques that
$\bar{v}_0$ is part of.

We then pick one of the cliques added to $\bar{\mathcal{R}}^{(n,\gamma
)}$ whose corresponding litter was not ignored in $\mathcal
{Z}^{(n,\gamma)}$. We realise a local epidemic in this group as
follows. Assume that the $B$-weight of the clique is $\bar{b}_1$. Then
let $d'_1$ be $\Poisson(\bar{b}_1 L^{(n)}/(\mu n))$.
Consider a population with $d'_1$ initial susceptible individuals and 1
initial infectious individual, all with infectious period distributed
as $\mathcal{I}(c_1)$, and couple two continuous time epidemics in this
population as follows. Consider the first newly infected individual in
this population. We associate this individual with vertices in $\bar
{\mathcal{R}}^{(n,\gamma)}$ and in $\mathcal{Z}^{(n,\gamma)}$ as
follows. Choose a real number, say $x$, uniformly at random from the
unit interval.
In $\bar{\mathcal{R}}^{(n,\gamma)}$, we try to connect clique $\bar
{v}'_1$ to the vertex with label $i$, which satisfies
\[
\sum_{j \in\mathbb{N}\dvtx \sigma^{(n)}_0(j) < \sigma^{(n)}_0(i)} A_{j} < x L^{(n)}
\leq\sum_{j \in\mathbb{N}\dvtx \sigma^{(n)}_0(j) \leq\sigma
^{(n)}_0(i)} A_{j}.
\]
Suppose that this vertex is $\bar{v}_2$.
The $A$-weight of the possible child in $\mathcal{Z}^{(n,\gamma)}$ is
$A_i$, where $i$ is such that
$\sum_{j = 1}^{i-1} A_j < x L^{(n)} \leq\sum_{j = 1}^{i} A_j$. The
vertex we choose is denoted by $\bar{v}_1$.
If $\bar{v}_1 \in\Gamma_0$, then the vertex is ignored in $\bar
{\mathcal{R}}^{(n,\gamma)}$ and immediately killed.
If $x > 1- \bar{\gamma}$, then the child in $\mathcal{Z}^{(n,\gamma)}$
is ignored.
We note that as long as the weight of $\Gamma_0$ is less than $\bar
{\gamma}L^{(n)}$, a vertex can be ignored in $\bar{\mathcal
{R}}^{(n,\gamma)}$ only if the child in $\mathcal{Z}^{(n,\gamma)}$ is
also ignored. Furthermore, the $A$-weight of the vertex in $\mathcal
{Z}^{(n,\gamma)}$ is not larger than the $A$-weight of the vertex in
$\bar{\mathcal{R}}^{(n,\gamma)}$.

We identify the other vertices infected by local epidemics started by
$v_0$ and the corresponding children in $\mathcal{Z}^{(n,\gamma)}$ as
we have identified the cliques $v_0$ is part of, where at each step the
forbidden set of vertices might grow and the chosen vertex gets the
highest label for the next vertex pick. The infectious period/type
assigned to every vertex (which is not immediately killed) is
distributed as $\mathcal{I}(c_1)$, and coupled vertices get the same
infectious period/type.
We continue in this way until we have identified all vertices infected
by local epidemics started by $v_0$, and we then explore the cliques
those individuals are part of one by one, as before.

The exploration process $\bar{\mathcal{R}}^{(n,\gamma)}$ dominates the
exploration process $\mathcal{Z}^{(n,\gamma)}$ until the total weight
of the forbidden set in $V^{(n)}$ in $\bar{\mathcal{R}}^{(n,\gamma)}$
is at least $\bar{\gamma} L^{(n)}$ or the total weight of the forbidden
set in $V'^{(n)}$ in $\bar{\mathcal{R}}^{(n,\gamma)}$ is at least $\bar
{\gamma} L'^{(n)}$.

Note that we may choose $c_1>0$ small enough such that $\mathbb
{P}(\mathcal{I} < c_1) < \gamma/2$. By the law of large numbers this
implies that $c_1>0$ might be chosen such that the total weight of
vertices in $I(c_1)$ is less than $(\gamma/2) L^{(n)}$ with probability
tending to 1 as $n\to\infty$.
By Lemma $\ref{weightsus}$, we know that the weights of $K^1$, $K^2$,
$K'^1$ and $K'^2$ are each a.s. $o(n)$. We also know that the set of
vertices with label $\geq R'^{(n)}(\gamma/3)$ has total weight at
least $(\gamma/3) L^{(n)}$, and the probability that this total weight
is less than $(\gamma/2) L^{(n)}$ can be made arbitrary close to 1
by choosing $n$ sufficiently large.

If the ordering of the exploration processes $\bar{\mathcal
{R}}^{(n,\gamma)}$ and $\mathcal{Z}^{(n,\gamma)}$ stops because the
total weight of the forbidden set in $V'^{(n)}$ exceeds $\gamma
L'^{(n)}$, then, using Lemma~\ref{killprocess}, the lemma is immediate
with $\eta= \gamma/3$. If this ordering stops because the total weight
of the forbidden set in $V^{(n)}$ exceeds $\gamma L^{(n)}$, then the
total weight of vertices in
$\bar{\mathcal{R}}^{(n,\gamma)}$ that are not in the original forbidden
set exceeds $(\gamma/3) L^{(n)}$.
Hence, in order to prove the lemma we have only to prove that this
implies that the total weight of cliques in this set which contain
vertices in $\bar{\mathcal{R}}^{(n,\gamma)}$ is $\Theta_{p_\nu}(n)$.
Now the fraction of cliques in $V'^{(n)}$ with weight exceeding $\log
n$ converges almost surely to $0$ as $n \to\infty$. It follows that
the number of vertices in $\hat{V}'^{(n)}$ with labels exceeding
$R'^{(n)}(\gamma/3)$ is $\Theta_{p_\nu}(n)$. Hence, by the law of large
numbers, the number of cliques $\hat{V}'^{(n)}$ with labels exceeding
$R'^{(n)}(\gamma/3)$ that are chosen for the expansion of $\bar{\mathcal
{R}}^{(n,\gamma)}$ is $\Theta_{p_\nu}(n)$, which in turn implies that
the total weight of such cliques is also $\Theta_{p_\nu}(n)$, as required.
\end{pf}

\begin{pf*}{Proof of Theorem \ref{finalsizethm}}
We use the notation of Lemma \ref{eplarbra}.
Recall that $\bar{\mathcal{R}}^{(n)} = \bar{\mathcal{R}}^{(n,0)}$ and
that $\mathcal{R}^{(n)} = \mathcal{R}^{(n)}(\omega,\mathcal{I})$ is the
set of ultimately recovered vertices in a population of $n$ individuals.

We first provide bounds for
\begin{eqnarray*}
\mathbb{E}_\omega\Bigl[n^{-1}\bigl|\mathcal{R}^{(n)}\bigr| \Bigm|
\bar{\mathcal {W}}'^{(n)}(c_1)> \eta n\Bigr] &
= & \mathbb{E}_\omega\Biggl[n^{-1}\sum
_{i=1}^n \ind\bigl(v_i \in\mathcal
{R}^{(n)}\bigr)\Bigm| \bar{\mathcal{W}}'^{(n)}(c_1)>
\eta n\Biggr]
\\
& = & \mathbb{P}_{\omega}\bigl(v_1 \in\mathcal{R}^{(n)}
\bigm| \bar {\mathcal{W}}'^{(n)}(c_1)> \eta n
\bigr)
\end{eqnarray*}
and for
\begin{eqnarray*}
&&\mathbb{E}_\omega \Bigl[n^{-2}\bigl|\mathcal{R}^{(n)}\bigr|^2
\Bigm| \bar{\mathcal {W}}'^{(n)}(c_1)> \eta n
\Bigr]
\\
& &\qquad= \mathbb{E}_\omega\Biggl[n^{-2}\sum
_{i=1}^n \sum_{j=1}^n
\ind\bigl(v_i,v_j \in \mathcal{R}^{(n)}\bigr) \Bigm|
\bar{\mathcal{W}}'^{(n)}(c_1)> \eta n\Biggr]
\\
& &\qquad= n^{-1}\mathbb{P}_{\omega}\bigl(v_1 \in
\mathcal{R}^{(n)} \bigm| \bar {\mathcal{W}}'^{(n)}(c_1)>
\eta n\bigr)
\\
& &\qquad\quad{}+ \bigl(1-n^{-1}\bigr)\mathbb{P}_{\omega}
\bigl(v_1,v_2 \in\mathcal{R}^{(n)} \bigm| \bar{
\mathcal{W}}'^{(n)}(c_1)> \eta n\bigr).
\end{eqnarray*}
Let $\varepsilon'>0$.
By Lemma \ref{sizesus} and the asymptotic theory of supercritical
general branching processes \cite{Nerm81} modified to the lattice case,
we have that, if the susceptibility set of $v_1$ in $\hat{G}^{(n)}$
survives for $t_n=\lceil\log\log n\rceil$ generations, then there
exists $c_2>0$ such that the probability that the number and the total
weight of the vertices in this generation are both at least $c_2\log\log n$
is greater than $1-\varepsilon'$ for all sufficiently large $n$. We denote
the set of vertices in generation $t_n$ of this susceptibility set by
$\hat{V}_{t_n}^{(n)}$.
The same holds for the susceptibility set of $v_2$. Furthermore, the
events of survival up to generation $t_n$ of the two susceptibility
sets are asymptotically independent by a birthday problem type of
argument and Lemma \ref{weightsus}.

Conditioned on $\bar{\mathcal{W}}'^{(n)}(c_1)> \eta n$, the law of
large numbers establishes that the following event occurs with
probability exceeding $1-\varepsilon'$. The number of vertices in $\hat
{V}_{t_n}^{(n)}$ that both (i) are in the same clique as an infected
vertex explored in $\bar{\mathcal{R}}^{(n)}$ and (ii) have infectious
period at least $c_1$, grows to infinity as $n \to\infty$.
Since each vertex in $\hat{V}_{t_n}^{(n)}$ is infected independently
with probability at least $1-\mathrm{e}^{-c_1}>0$, we have that
\[
\ind \Bigl(\mathbb{P}_{\omega} \Bigl(v_1 \in
\mathcal{R}^{(n)} \Bigm| \bigl|\hat{\mathcal{S}}^1_{t_n}\bigr|>0,
\bar{\mathcal{W}}'^{(n)}(c_1)> \eta n
\Bigr)>1-2 \varepsilon' \Bigr) \parrow1.
\]
Furthermore, if the susceptibility set of $v_1$ does not survive up to
generation $t_n$ in $\hat{G}^{(n)}$, then Lemma \ref{lemignorehat}
shows that the probability that the initial infective is in $v_1$'s
susceptibility set converges to 0. More precisely, for every $K \in
\mathbb{N}$ we have that
\begin{eqnarray*}
&&\mathbb{P}_{\omega} \Bigl(v_1 \in\mathcal{R}^{(n)}
\Bigm| \bigl|\hat {\mathcal{S}}^1_{t_n}\bigr|=0 \Bigr)\\
&&\qquad =
\frac{\mathbb{P}_{\omega}(v_1 \in\mathcal{R}^{(n)},|\hat{\mathcal
{S}}^1_{t_n}|=0)}{\mathbb{P}_{\omega}(|\hat{\mathcal{S}}^1_{t_n}|=0)}
\\
&&\qquad\leq\frac{\mathbb{P}_{\omega}(v_1 \in\mathcal{R}^{(n)},|\mathcal
{S}^1|\leq K) + \mathbb{P}_{\omega}(|\mathcal{S}^1| > K,|\hat{\mathcal
{S}}^1_{t_n}|=0)}{\mathbb{P}_{\omega}(|\hat{\mathcal{S}}^1_{t_n}|=0)}.
\end{eqnarray*}
The first term in the numerator of the right-hand side of this
inequality converges in probability to 0 as $n \to\infty$, while by
Lemma \ref{lemignorehat} we have that, for every $\varepsilon>0$ and
$\delta>0$, there exists $K \in\mathbb{N}$ such that the second term
in the numerator is smaller than $\varepsilon$ with $\nu$-probability at
least $1-\delta$ for all sufficiently large $n$.
Clearly $\liminf_{n\to\infty} \mathbb{P}_{\omega}
(|\hat{\mathcal {S}}^1_{t_n}| = 0) > 0.$
We therefore conclude that
\[
\mathbb{P}_{\omega}\Bigl(v_1 \in\mathcal{R}^{(n)} \Bigm|
\bigl|\hat{\mathcal {S}}^1_{t_n}\bigr|=0\Bigr)\parrow0.
\]

Note that in the proof of Lemma \ref{eplarbra} we do not use whether $
|\hat{\mathcal{S}}^1_{t_n}|>0$ or not, so
\[
\mathbb{P}_{\omega} \Bigl(\bar{\mathcal{W}}'^{(n)}(c_1)>
\eta n  \Bigm|\bigl |\hat{\mathcal{S}}^1_{t_n}\bigr|>0 \Bigr)-
\mathbb{P}_{\omega} \Bigl(\bar{\mathcal{W}}'^{(n)}(c_1)>
\eta n  \Bigm|\bigl |\hat{\mathcal{S}}^1_{t_n}\bigr|=0 \Bigr) \parrow0.
\]
Therefore, by Bayes's theorem, we find that
\[
\mathbb{P}_{\omega} \bigl(\bigl|\hat{\mathcal{S}}^1_{t_n}\bigr|>0
\bigr)- \mathbb{P}_{\omega} \Bigl(\bigl|\hat{\mathcal{S}}^1_{t_n}\bigr|>0
 \Bigm| \bar {\mathcal{W}}'^{(n)}(c_1)> \eta n
\Bigr) \parrow0,
\]
whence
\[
\mathbb{P}_{\omega} \Bigl(v_1 \in\mathcal{R}^{(n)}
 \Bigm| \bar{\mathcal {W}}'^{(n)}(c_1)> \eta n
\Bigr)- \mathbb{P}_{\omega} \bigl(\bigl|\hat{\mathcal{S}}^1_{t_n}\bigr|>0
\bigr) \parrow0.
\]
Now, arguing as at the start of the proof of Lemma \ref{sizesus},
\[
\mathbb{P}_{\omega} \bigl(\bigl|\hat{\mathcal{S}}^1_{t_n}\bigr|>0
\bigr)- \mathbb{P}_{\omega} \bigl(\bigl|\mathcal{Z}^{b}_{t_n}
\bigl(A^{(n)},B^{(n)},\mathcal {I}\bigr)\bigr|>0\bigr) \parrow0,
\]
while the end of the proof of Lemma \ref{sizesus} shows that
\[
\mathbb{P}_{\omega} \bigl(\bigl|\mathcal{Z}^{b}_{t_n}
\bigl(A^{(n)},B^{(n)},\mathcal {I}\bigr)\bigr|>0\bigr) \parrow
\rho^b(A,B,\mathcal{I}).
\]
Thus, $\mathbb{P}_{\omega}(|\hat{\mathcal{S}}^1_{t_n}|>0) \parrow\rho
^b(A,B,\mathcal{I})$,
whence
\[
\mathbb{E}_\omega\Bigl[n^{-1}\bigl|\mathcal{R}^{(n)}\bigr|  \Bigm|
\bar{\mathcal {W}}'^{(n)}(c_1)> \eta n\Bigr]
\parrow\rho^b(A,B,\mathcal{I}).
\]

Since the first $t_n$ generations of the susceptibility sets of $v_1$
and $v_2$ in $\hat{G}^{(n)}$ are nonoverlapping with probability
tending to 1 as $n \to\infty$, we notice that
\[
\mathbb{P}_{\omega} \bigl(v_1,v_2 \in
\mathcal{R}^{(n)}  \bigm| \bar {\mathcal{W}}'^{(n)}(c_1)>
\eta n\bigr)- \bigl(\mathbb{P}_{\omega} \bigl(v_1 \in
\mathcal{R}^{(n)}  \bigm| \bar{\mathcal{W}}'^{(n)}(c_1)>
\eta n \bigr)\bigr)^2
\parrow0.
\]
This gives that
\[
\mathbb{E}_\omega\Bigl[n^{-2}\bigl|\mathcal{R}^{(n)}\bigr|^2
 \Bigm| \bar{\mathcal {W}}'^{(n)}(c_1)> \eta n
\Bigr] \parrow\bigl(\rho^b(A,B,\mathcal{I})\bigr)^2.
\]
Therefore, $\mathrm{var} (n^{-1}|\mathcal{R}^{(n)}|   \mid  \bar{\mathcal
{W}}'^{(n)}(c_1)> \eta n) \parrow0$,
and we conclude that, for all $\delta>0$,
%
%
\begin{equation}
\label{finimp1} \mathbb{P}_{\omega} \Bigl(\bigl|n^{-1}
\mathcal{R}^{(n)} - \rho^b(A,B,\mathcal {I})\bigr| < \delta  \Bigm|
\bar{\mathcal{W}}'^{(n)}(c_1)> \eta n\Bigr)
\parrow1.
\end{equation}

On the other hand,
we know by Lemma \ref{eplarbra} that
for every $\varepsilon'>0$, there exist constants $\eta>0$ and $c_1>0$
such that
%
%
\begin{equation}
\label{boundneed1} \ind\bigl(\mathbb{P}_{\omega}\bigl(\bar{
\mathcal{W}}'^{(n)}(c_1)> \eta n\bigr) > \rho
(A,B,\mathcal{I}) -\varepsilon'\bigr) \parrow1.
\end{equation}
Furthermore, by Theorem \ref{extinctionthm} there exists $k \in\mathbb
{N}$ such that
%
%
\begin{equation}
\label{boundneed2} \ind \Biggl(\sum_{i=1}^k
\mathbb{P}_{\omega}\bigl(\bigl|\mathcal{R}^{(n)}\bigr|=i\bigr) > 1-
\rho(A,B,\mathcal{I}) -\varepsilon' \Biggr) \parrow1.
\end{equation}
Now observe that
\begin{eqnarray*}
&&\mathbb{P}_{\omega}\bigl(v_1 \in\mathcal{R}^{(n)},
\bar{\mathcal {W}}'^{(n)}(c_1) \leq\eta n
\bigr)\\
&&\qquad \leq \mathbb{P}_{\omega}\bigl(v_1 \in
\mathcal{R}^{(n)}, \bigl|\mathcal {R}^{(n)}\bigr| \leq k \bigr)
 + \mathbb{P}_{\omega}\bigl(\bar{\mathcal{W}}'^{(n)}(c_1)
\leq\eta n, \bigl|\mathcal{R}^{(n)}\bigr| > k \bigr).
\end{eqnarray*}
By exchangeability, the first term on the right-hand side of this
inequality is bounded above by $k/n$ which converges to 0 as $n\to
\infty$.
Further, for any $K \in\mathbb{N}$,
\[
\mathbb{P}_{\omega}\bigl(\bar{\mathcal{W}}'^{(n)}(c_1)
> \eta n, \bigl|\mathcal {R}^{(n)}\bigr| \le K\bigr) \parrow0,
\]
so \eqref{boundneed1} and \eqref{boundneed2} imply that for every
$\varepsilon>0$, there exists $k\in\mathbb{N}$ such that
\[
\ind\bigl(\mathbb{P}_{\omega}\bigl(\bar{\mathcal{W}}'^{(n)}(c_1)
\leq\eta n, \bigl|\mathcal{R}^{(n)}\bigr| > k\bigr)< \varepsilon\bigr) \parrow1.
\]
It follows that
\[
\mathbb{E}_\omega\Bigl[n^{-1}\bigl|\mathcal{R}^{(n)}\bigr|  \Bigm|
\bar{\mathcal {W}}'^{(n)}(c_1) \leq\eta n
\Bigr] \parrow0,
\]
so for every $\delta>0$ we have
%
%
\begin{equation}
\label{finimp2} \mathbb{P}_{\omega} \Bigl(n^{-1}\bigl|
\mathcal{R}^{(n)} \bigr| < \delta  \Bigm| \bar {\mathcal{W}}'^{(n)}(c_1)
\leq\eta n\Bigr) \parrow1.
\end{equation}
Combining \eqref{finimp1} and \eqref{finimp2} completes the proof of
Theorem \ref{finalsizethm}.
\end{pf*}

\section{Extension}\label{disc}

In this paper we study the spread of an SIR epidemic on a random
intersection graph. A variant of the random intersection graph is
proposed in \cite{Newm03}, where a configuration model construction is
used to create the graph. In our terminology and notation, independent
degrees are assigned to vertices in $V$ and~$V'$, where the degrees of
vertices in $V$ are each distributed as a random variable~$D$, and the
degrees of vertices in $V'$ are each distributed as a random variable~$H$. Each vertex in $V \cup V'$ is assigned a number of half-edges
given by its degree. In the auxiliary graph $\mathbb{A}^{(n)}$ the
half-edges of the first $n$ vertices in $V$ are paired uniformly at
random with the first $L^{(n)}$ half-edges in $V'$, where $L^{(n)}$ is
the number of half-edges assigned to the first $n$ vertices in $V$.
Note that the final vertex in $V'$ used in this construction might not
retain its full degree in $\mathbb{A}^{(n)}$.

The forward and backward branching processes can be modified in the
obvious fashion to this setting and equivalent formulae to the key
expressions \eqref{1minfh}, \eqref{elocal} and \eqref{suslocal} in
Appendix \ref{application} can be derived, thus facilitating
calculation of the threshold parameter $R_*$ and survival probabilities
of these branching processes. We expect that, under mild conditions on
the distributions of $D$ and $H$, theorems corresponding to
Theorems \ref{extinctionthmu}--\ref{finalsizecdthm} hold for this model.
Some additional dependencies arise since connecting to a vertex takes
away one of its available half-edges; however, we anticipate that the
impact of those dependencies is very small.

\begin{appendix}\label{app}
\section{Proof of Lemma \lowercase{\protect\texorpdfstring{\ref{unilem}}{4.1}}}
\label{appendix1}

In order to prove Lemma \ref{unilem} we use an idea from Riordan \cite
{Rior05}. He considers the corresponding problem for a class of
multitype branching processes having type space $(0,1]$ in which, in
crude terms, the number of children having type in any specified
interval spawned by an individual of type $x$ tends to infinity as $x
\downarrow0$. We cannot use the result in \cite{Rior05} directly
because in our model the number of children of an individual of type $x$ tends to
zero as $x \downarrow0$. However, we can apply the idea in \cite
{Rior05} to a branching process that is intimately related to $\tilde
{\mathcal{Z}}^f$, which we now describe, and exploit a connection
between the functional $\Phi(\tilde{\rho})(x)$ and an equivalent
functional for the new branching process to obtain the desired result.

Recall that in the branching process $\tilde{\mathcal{Z}}^f$,
individuals arise in litters, with a litter being distributed as the
set of individuals that are infected in a local (single-clique)
epidemic, not including the individual who triggers that local
epidemic. Consider such a local epidemic and suppose that the clique
contains the initial infective, $i^*$ say, and $m$ susceptible
individuals. The final outcome of the local epidemic can be obtained
using the corresponding epidemic generated graph, by first determining
the number of individuals, $a$ say, that are contacted directly by the
initial infective, and then considering the epidemic, $\mathcal
{E}_{s,a}$ say, triggered by those $a$ individuals among the remaining
$s=m-a$ susceptibles in the clique. Suppose that the epidemic $\mathcal
{E}_{s,a}$ infects $T_{s,a}$ individuals, in addition to its $a$
initial infectives.
[Thus, in the notation of Section~\ref{earlystages}, $T(m)=a+T_{s,a}$.] Note that the infectious periods of
the $a$ initial infectives in $\mathcal{E}_{s,a}$ are i.i.d. copies of
$\mathcal{I}$ and also that, conditional upon the value of $(s,a)$,
such epidemics in different cliques are mutually independent, even if
they arise from the same initial infective $i^*$. Thus the epidemic
$\mathcal{E}^{(n)}$ may be approximated by a branching process of
litters, in which each litter is typed by its value of $(s,a)$ and its
offspring are the litters triggered by the $a+T_{s,a}$ infectives in
the corresponding $\mathcal{E}_{s,a}$. Let $\hat{\mathcal{Z}}^f$ be the
branching process derived in this fashion corresponding to the
branching process $\tilde{\mathcal{Z}}^f$. Clearly, litters with $a=0$
are superfluous, so the type space for $\hat{\mathcal{Z}}^f$ may be
taken to be $\hat{\mathcal{T}}=\{(s,a)\dvtx s \in\mathbb{Z}_+,a \in\mathbb
{N}\} $.

We now derive the next-generation functional [i.e., the analogue of
$\tilde{\Phi}(h)(x)$] associated with $\hat{\mathcal{Z}}^f$.
For notational convenience we assume that $\mathcal{I}$ has an
absolutely continuous distribution, though this is not essential, and
the argument (and the proof of Lemma \ref{unilem} below) can be
extended to the general case.
Let $\hat{h}(s,a)\dvtx \hat{\mathcal{T}} \to[0,1]$ be a measurable test function,
and suppose that litters are marked independently with a dagger (to
distinguish from the marks used on $\mathcal{Z}^f$), with a litter of
type $(s,a)$ being marked with probability $\hat{h}(s,a)$. Let $\hat
{\Phi}(\hat{h})(s,a)$ be the probability that a litter of type $(s,a)$
directly spawns at least one litter that is marked with a dagger.

Consider the epidemic $\mathcal{E}_{s,a}$ described above and suppose
that $T_{s,a}=k$.
Let $x_{-a+1},x_{-a+2}, \ldots, x_0$ and $x_1, x_2,\ldots, x_k$ denote
the lengths of the infectious periods of the $a$ initial infectives and
the $k$ subsequently infected individuals, respectively.
Let $p_{s,a}(k; x_{-a+1},x_{-a+2}, \ldots, x_0, x_1,\ldots, x_k)$ be
the probability density that $T_{s,a}=k$ and the infectious periods are
given by
$x_{-a+1},\ldots, x_k$. Then,
%
%
\begin{eqnarray}
\label{Phifirst} \hat{\Phi}(\hat{h}) (s,a)&=&1- \sum_{k=0}^s
\int_{(0,\infty]^{a+k}} p_{s,a}(k;x_{-a+1},\ldots,
x_k)
\nonumber
\\[-8pt]
\\[-8pt]
\nonumber
&&\hspace*{75pt}{}\times\prod_{i=-a+1}^k
P_{\hat{h}}(x_i) \,\mathrm{d}x_{-a+1}\cdots\mathrm{d}x_{k},
\end{eqnarray}
where $P_{\hat{h}}(x)$ is the probability that an individual, $i^*$
say, having infectious period of length $x$, does not spawn a litter
which is marked with a dagger.

To determine $P_{\hat{h}}(x)$, note first that $i^*$ belongs to $\check
{X}\sim\mathcal{MP}(\tilde{A})$ cliques, not counting the clique it was
infected through, and consider one such clique. Besides $i^*$, this
clique contains $\check{Y}\sim\mathcal{MP}(\tilde{B})$ individuals.
Suppose that $\tilde{B}=b$, then $\check{Y}\sim\mathcal{P}(b)$ and
these $\check{Y}$ individuals are infected independently by $i^*$, each
with probability
$1-\mathrm{e}^x$. Thus, given $\tilde{B}=b$, the litter has type $(s,a)$, where
$s$ and $a$ are independent realisations of the Poisson random
variables $\mathcal{P}(b\mathrm{e}^x)$ and $\mathcal{P}(b(1-\mathrm{e}^x))$,
respectively. Hence, the unconditional probability
that this litter is not marked with a dagger is
\[
\mathbb{E} \Biggl[ \sum_{s=0}^{\infty} \sum
_{a=0}^{\infty} \frac{(\mathrm{e}^{-x}\tilde{B})^{s}}{s!}
\frac{((1-\mathrm{e}^{-x})\tilde{B})^{a}}{a!} \mathrm{e}^{-\tilde{B}} \bigl(1-\hat{h}(s,a)\bigr) \Biggr],
\]
where $\hat{h}(s,0)=0$ $(s \in\mathbb{Z}_+)$. Given that $i^*$ has
infectious period $x$, the local epidemics it initiates in the above
$\check{X}$ cliques are independent, so
%
%
\begin{equation}
\label{Phhatx} P_{\hat{h}}(x)=\phi_{\tilde{A}} \Biggl(\mathbb{E}
\Biggl[ \sum_{s=0}^{\infty
} \sum
_{a=0}^{\infty} \frac{(\mathrm{e}^{-x}\tilde{B})^{s}}{s!} \frac
{((1-\mathrm{e}^{-x})\tilde{B})^{a}}{a!} \mathrm{e}^{-\tilde{B}}\hat {h}(s,a) \Biggr] \Biggr).
\end{equation}

Let $\hat{\rho}(s,a)$ be the survival probability of the branching
process $\hat{\mathcal{Z}}^f$, given that the initial litter has type
$(s,a)$. Then $\hat{\rho}$ is the maximal solution of $\hat{\rho}(s,a)
= \hat{\Phi}(\hat{\rho})(s,a)$. If either $s \to\infty$ or $a \to
\infty$, then for any $(s',a') \in\hat{\mathcal{T}}$ and any $K\in
\mathbb{N}$, the probability that a type-$(s,a)$ individual has at
least $K$ type-$(s',a')$ children in the next generation tends to 1.
Furthermore, it is easy to deduce that for any $(s,a), (s',a') \in\hat
{\mathcal{T}}$, the number of type-$(s',a')$ children an individual of
type $(s,a)$ begets is nonzero with positive probability, so $\hat
{\mathcal{Z}}^f$ is irreducible. Using the same argument as
in pages 911--912 of \cite{Rior05}, we conclude that there is at most one
nonzero solution of
$\hat{\rho}(s,a) = \hat{\Phi}(\hat{\rho})(s,a)$.

Recall that Lemma \ref{unilem} states that there is at most one
nonzero solution $\tilde{\rho}(x)$ of the functional equation $\tilde
{\rho}(x) = \tilde{\Phi}(\tilde{\rho})(x)$.
To prove this it is useful to derive an alternative expression for
$\tilde{\Phi}(h)(x)$. Suppose that the ancestor, $i^*$ say, in
$\tilde{\mathcal{Z}}^f$ has infectious period of length $x$. By
conditioning on the size of and the number of people directly infected
by $i^*$ in a given clique, the probability that $i^*$ has no marked
child in that clique is given by
%
%
\begin{equation}
\label{nomarkchild} \mathbb{E} \Biggl[ \sum_{s=0}^{\infty}
\sum_{a=0}^{\infty} \frac{(\mathrm{e}^{-x}\tilde{B})^{s}}{s!}
\frac{((1-\mathrm{e}^{-x})\tilde{B})^{a}}{a!} \mathrm{e}^{-\tilde{B}} A(s,a,h) \Biggr],
\end{equation}
where
%
%
\begin{eqnarray}
\label{APhihelp} A(s,a,h)&=&\sum_{k=0}^s
\int_{(0,\infty]^{a+k}} p_{s,a}(k;x_{-a+1},\ldots,
x_k)
\nonumber
\\[-8pt]
\\[-8pt]
\nonumber
&&\hspace*{56pt}{}\times\prod_{i=-a+1}^k
\bigl(1-h(x_i)\bigr) \,\mathrm{d}x_{-a+1}\cdots\mathrm{d}x_{k}.
\end{eqnarray}
Hence, since $i^*$ belongs to $\check{X}\sim\mathcal{MP}(\tilde{A})$
further cliques (in addition to the one it was infected through),
%
%
\begin{eqnarray}
\label{Phitransform} \qquad &&\tilde{\Phi}(h) (x)
\nonumber
\\[-4pt]
\\[-12pt]
\nonumber
&&\qquad=1-\phi_{\tilde{A}} \Biggl(\mathbb{E}
\Biggl[ \sum_{s=0}^{\infty} \sum
_{a=0}^{\infty} \frac{(\mathrm{e}^{-x}\tilde
{B})^{s}}{s!} \frac{((1-\mathrm{e}^{-x})\tilde{B})^{a}}{a!} \mathrm{e}^{-\tilde{B}} \bigl(1-A(s,a,h)\bigr) \Biggr] \Biggr).
\end{eqnarray}

Suppose that
%
%
\begin{equation}
h(x) = \tilde{\Phi}(h) (x). \label{hPhitilde}
\end{equation}
Then \eqref{Phitransform} and \eqref{APhihelp} imply that
\begin{eqnarray*}
&&A(s,a,h)\\
&&\qquad= \sum_{k=0}^s \int
_{(0,\infty]^{a+k}} p_{s,a}(k;x_{-a+1},\ldots,
x_k)
\\
&&\qquad\quad{}\times\prod_{i=-a+1}^k \phi_{\tilde{A}}
\Biggl(\mathbb{E} \Biggl[ \sum_{s_i=0}^{\infty}
\sum_{a_i=0}^{\infty} \frac{(\mathrm{e}^{-x_i}\tilde
{B})^{s_i}}{s_i!}
\frac{((1-\mathrm{e}^{-x_i})\tilde{B})^{a_i}}{a_i!}\\
&&\qquad\quad\hspace*{128pt}{}\times \mathrm{e}^{-\tilde{B}} \bigl(1-A(s_i,a_i,h)
\bigr) \Biggr] \Biggr)\,
\mathrm{d}x_{-a+1}\cdots\mathrm{d}x_{k}.
\end{eqnarray*}
Thus, by \eqref{Phifirst} and \eqref{Phhatx}, if $h$ is treated as
fixed, $\hat{h}(s,a)=1-A(s,a,h)$ satisfies
%
%
\begin{equation}
\hat{h}(s,a)=\hat{\Phi}(\hat{h}) (s,a). \label{hPhihat}
\end{equation}
Let $h$ be a nonzero (i.e., not identically zero) solution of \eqref
{hPhitilde}, assuming such a solution exists. Then $\hat{h}$ must be
the unique nonzero solution of \eqref{hPhihat}, $\hat{\rho}$ say.
[Note that if $\hat{h}$ is identically zero, then \eqref{Phitransform}
and \eqref{hPhitilde} imply that $h$ is identically zero.] Thus $\hat
{h}(s,a)=1-A(s,a,h)$ is independent of $h$, and $h(x)$ is given by the
right-hand side of \eqref{Phitransform} with $A(s,a,h)$ replaced by
$1-\hat{\rho}(s,a)$, which proves the lemma.


\section{Calculation of properties of forward and backward branching processes}
\label{FSRV}

In this appendix we give expressions for properties of the forward and
backward branching processes, $\mathcal{Z}^f$ and $\mathcal{Z}^b$,
which enable the threshold parameter $R_*$ and the survival
probabilities $\rho$ and $\rho^b$ which appear in Theorem \ref{finalsizecdthm} to be computed. These expressions rest on results for
the final outcome of homogeneously mixing SIR epidemic models. In a
series of papers (see, e.g., \cite{Picard90}), Lef\`{e}vre and Picard
showed that many quantities related to the final outcome of an SIR
epidemic can be expressed compactly in terms of Gontcharoff
polynomials, and these were extended by Ball and O'Neill \cite{Ball99}
to include so-called general final state random variables. The latter
are required to compute functionals associated with the forward
branching process $\mathcal{Z}^f$. Results for
homogeneously mixing SIR epidemic models are outlined in Section~\ref{RHMP}
and their application to computing properties of $\mathcal{Z}^f$ and
$\mathcal{Z}^b$ is described in Section~\ref{application}.

\subsection{Results for homogeneously mixing populations}
\label{RHMP}
In this section we give a restatement of Theorem 4.2 from Ball and
O'Neill \cite{Ball99}, adapted to the purposes of this paper (cf. \cite
{Ball10}). We note that Ball and O'Neill provide appreciably more
general results than their Theorem 4.2.
In order to state the theorem, we need the following notation.
We consider an SIR epidemic in a homogeneously mixing population with
$s$ initial susceptible individuals and $a$ initial infectious
individuals. The initial susceptible individuals are labeled $1,2,
\ldots,s$ and the initial infectious individuals have labels
$-a+1,-a+2,\ldots, 0$. The random variable $\mathcal{I}_i$ represents
the infectious period that individual $i$ will have if it becomes infected.
Thus, the probability that individual $i$, if infected, ultimately has
an infectious contact with individual $j$ is $1-\mathrm{e}^{-\mathcal
{I}_i}$. (As before, infectious contacts between pairs of individuals
are governed by independent unit-rate Poisson processes.) We assume
that the random variables $(\mathcal{I}_i, i = -a+1,-a+2,\ldots, s)$
are independent and all distributed as $\mathcal{I}$; they are also
independent of the Poisson processes describing infectious contacts.
Note that this model is the epidemic $\mathcal{E}_{s,a}$ introduced in
Appendix~\ref{appendix1}.
Let $\hat{h}(x)\dvtx (0,\infty] \to[0,\infty]$ be a measurable function
(the relevant measures are clear from the context) and $\theta>0$.
Furthermore, let
\[
\hat{U}=\hat{U}(\hat{h},\theta) = \bigl(\hat{u}_i(\hat{h},\theta), i
\in \mathbb{Z}_+\bigr) = (\hat{u}_i, i \in\mathbb{Z}_+)
\]
be an infinite vector, where $\hat{u}_k = \mathbb{E}[\mathrm{e}^{-k
\mathcal{I}} \mathrm{e}^{-\theta\hat{h}(\mathcal{I})}]$.
Let $\mathcal{R}$ be the set of ultimately recovered individuals in
$\mathcal{E}_{m,a}$, including the initial infectives as well as any
initial susceptibles that become infected.

The Gontcharoff polynomials $G_m(x|\hat{U}), m \in\mathbb{Z}_+$, are
defined recursively by
%
%
\begin{equation}
\label{gonrecurs} \frac{x^m}{m!}= \sum_{k=0}^{m}
\frac{(\hat{u}_k)^{m-k}}{(m-k)!}G_k(x|\hat{U})
\end{equation}
for $m \in\mathbb{Z}_+$. We note that $G_m(x|\hat{U})$ is a polynomial
of order $m$, which depends on $\hat{u}_0,\hat{u}_1, \ldots, \hat{u}_{m-1}$.
Some properties of Gontcharoff polynomials are mentioned in Section~2
of \cite{Ball99}. In this paper we use only \eqref{gonrecurs} and
%
%
\begin{equation}
\label{goninteg} G_m(x|\hat{U})=\int_{\hat{u}_0}^{x}
\int_{\hat{u}_{1}}^{\xi_{0}} \cdots \int_{\hat{u}_{m-1}}^{\xi_{m-2}}
\,\mathrm{d}\xi_{m-1}\cdots\mathrm{d}\xi_1 \,\mathrm{d}
\xi_0
\end{equation}
for $m \in\mathbb{Z}_+$. The following theorem is a special case of
Theorem 4.2 in \cite{Ball99}, which allows $\hat{h}$ to be random.

\begin{theorem}\label{thmfrankphil}
For $\mathcal{R}$, $\hat{h}$ and $\hat{U}$ as above, we have
\[
\mathbb{E}\bigl[x^{s+a-|\mathcal{R}|}\mathrm{e}^{-\theta\sum_{i \in\mathcal
{R}} \hat{h}(\mathcal{I}_i)}\bigr] = \sum
_{k=0}^s \frac{s!}{(s-k)!}(\hat{u}_k)^{s-k+a}G_k(x|
\hat{U}).
\]
\end{theorem}

We use the following corollary of this theorem.
%
\begin{corollary}\label{corfrankphil}
Let $U= U(h) =(u_i(h), i \in\mathbb{Z}_+) =(u_i, i \in\mathbb{Z}_+)$,
where $u_i =\mathbb{E}[\mathrm{e}^{-i \mathcal{I}}(1-h(\mathcal{I}))]$ and
$h(x)\dvtx (0,\infty] \to[0,1]$ is Borel-measurable, and let $\mathcal{R}$
be as above. Then
%
%
\begin{equation}
\label{coreq} \mathbb{E}\biggl[\prod_{i \in\mathcal{R}}
\bigl(1-h(\mathcal{I}_i)\bigr)\biggr]= \sum
_{k=0}^s \frac{s!}{(s-k)!} (u_k)^{s-k+a}G_k(1|U).
\end{equation}
\end{corollary}
\begin{pf} Set $x=\theta=1$ and $\hat{h} = -\log(1-h)$ in Theorem \ref{thmfrankphil}.
\end{pf}

Recall the random variable $T(m)$ introduced in Section~\ref{earlystages}. In the present notation, $T(m)$ is the size of the
epidemic $\mathcal{E}_{m,1}$, not including the initial infective.
The\vadjust{\goodbreak}
mean of $T(m)$ can be expressed in terms of Gontcharoff polynomials as
follows (see, e.g., equation (3.6) of \cite{Ball97}):
%
%
\begin{eqnarray}
\label{meansize} \mathbb{E}\bigl[T(m)\bigr]=m-\sum_{k=1}^m
\frac{m!}{(m-k)!} (v_{k-1})^{m+1-k} G_{k-1}(1|V)
\nonumber
\\[-10pt]
\\[-10pt]
\eqntext{(m=1,2,\ldots),}
\end{eqnarray}
where $v_{k} = \mathbb{E}[\mathrm{e}^{-(k+1) \mathcal{I}}]$ and $V=(v_i,
i\in\mathbb{Z}_+)$.

The distribution of the size of the local susceptibility set of an
individual can also be expressed using Gontcharoff polynomials.
Recall from Section~\ref{finaloutcome} that $S(m)$ is the size of the
local susceptibility set of an individual in a clique of size $m+1$,
where $S(m)$ does \emph{not} include the individual in question.
As in Section~3 of~\cite{Ball10}, we have
%
%
\begin{equation}
\label{backwdist} \mathbb{P}\bigl(S(m)=k\bigr) = \frac{m!}{(m-k)!}
(v_k)^{m-k} G_k(1|V)\qquad (k=0,1,\ldots, m),
\end{equation}
where $v_k$ and $V$ are as in \eqref{meansize}.

\subsection{Application to branching processes $\mathcal{Z}^f$ and
$\mathcal{Z}^b$}
\label{application}
Let $h$ and $U= U(h)$ be as in Corollary \ref{corfrankphil} and suppose
that individuals in
$\mathcal{E}_{s,a}$ are marked independently, with individual $i$ being
marked with probability $h(\mathcal{I}_i)$ $(i=-a+1,-a+2,\ldots, s)$.
Then \eqref{coreq} gives the probability that the epidemic $\mathcal
{E}_{s,a}$ contains no marked infective. Recall from Section~\ref{forwardBP}
that $F(h)(x)$ is the probability that the ancestor
in $\mathcal{Z}^f$ has at least one marked child arising from the local
epidemic in a given clique. Arguing as in the derivation of \eqref
{nomarkchild} gives, after repeatedly using Fubini's theorem [note that
$G_k(1|U)\ge0$ for all $k$, using \eqref{goninteg} and the fact that
$(u_k \in[0,1])$ is decreasing in $k$],
%
%
\begin{eqnarray}
\label{1minfh}
&&1-F(h) (x) \nonumber\\[-1pt]
&&\qquad = \mathbb{E} \Biggl[ \sum
_{s=0}^{\infty}\sum_{a=0}^{\infty}
\frac{\mathrm{e}^{-xs}\tilde{B}^s}{s!} \frac{(1-\mathrm{e}^{-x})^a \tilde{B}^a}{a!}\mathrm{e}^{-\tilde{B}}\sum
_{k=0}^s \frac{s!}{(s-k)!} (u_k)^{s-k+a}
G_k(1|U) \Biggr]\hspace*{-15pt}
\nonumber
\\[-1pt]
 &&\qquad =  \mathbb{E} \Biggl[\sum_{s=0}^{\infty}
\sum_{k=0}^s \frac{\mathrm{e}^{-xs} \tilde{B}^s}{s!}
\frac{s!}{(s-k)!} (u_k)^{s-k} G_k(1|U) \mathrm{e}^{-\tilde{B}(1-u_k(1-\mathrm{e}^{-x}))} \Biggr]
\nonumber
\\[-8pt]
\\[-8pt]
\nonumber
 &&\qquad =  \mathbb{E} \Biggl[\sum_{k=0}^{\infty}
\sum_{s=k}^{\infty} \frac
{\mathrm{e}^{-xs} \tilde{B}^s}{(s-k)!}
(u_k)^{s-k} G_k(1|U) \mathrm{e}^{-\tilde{B}(1-u_k(1-\mathrm{e}^{-x}))}
\Biggr]
\\[-1pt]
 &&\qquad =  \mathbb{E} \Biggl[ \sum_{k=0}^{\infty}
\mathrm{e}^{-xk} \tilde{B}^k \mathrm{e}^{-\tilde{B}(1-u_k)}
G_k(1|U) \Biggr]
\nonumber
\\[-1pt]
 &&\qquad =  \sum_{k=0}^{\infty} \bigl(-\mathrm{e}^{-x}\bigr)^k \phi_{\tilde
{B}}^{(k)}(1-u_k)
G_k(1|U),\nonumber
\end{eqnarray}
where $\phi_{\tilde{B}}^{(k)}$ is the $k$th derivative of
$\phi_{\tilde{B}}$.\vadjust{\goodbreak}

Finally, we derive expressions for $\mathbb{E}_{\check{Y}}[\mathbb
{E}[T(\check{Y})\mid \check{Y}]]$ and $\mathbb{E}_{\check{Y}}[f_{S(\check
{Y})|\check{Y}}(s)]$, where $\check{Y}\sim\mathcal{MP}(\tilde{B})$,
which are required to compute $R_*$ and $\rho^b$; see \eqref{R*def}
and \eqref{rhob}, respectively.\vspace*{1pt} Recall that $(\check{Y}|\tilde{B}=b)\sim
\mathcal{P}(b)$ and
$\mathbb{E}[\check{Y}]=\mathbb{E}[\tilde{B}]$. Thus conditioning on
$\tilde{B}$ and using \eqref{meansize} yields
\begin{eqnarray*}
&&\mathbb{E}_{\check{Y}}\bigl[\mathbb{E}\bigl[T(\check{Y})\mid \check{Y}\bigr]
\bigr]\\
&&\qquad=\mathbb {E}[\tilde{B}]-\mathbb{E} \Biggl[\sum_{m=1}^{\infty}
\frac{\tilde
{B}^m}{m!}\mathrm{e}^{-\tilde{B}} \sum_{k=1}^m
\frac{m!}{(m-k)!} (v_{k-1})^{m+1-k} G_{k-1}(1|V)
\Biggr].
\end{eqnarray*}
Interchanging the order of summation then yields, after elementary
algebra, that
%
%
\begin{equation}
\label{elocal}\qquad \mathbb{E}_{\check{Y}}\bigl[\mathbb{E}\bigl[T(\check{Y})\mid
\check{Y}\bigr]\bigr]=\mathbb {E}[\tilde{B}]-\sum_{k=1}^{\infty}v_{k-1}(-1)^k
\phi_{\tilde
{B}}^{(k)}(1-v_{k-1}) G_{k-1}(1|V).
\end{equation}

Turning to the size of the local susceptibility set of an individual in
a typical clique, first note that conditioning on $\tilde{B}$ and
using \eqref{backwdist} gives, for $k\in\mathbb{Z}_+$,
\begin{eqnarray*}
\mathbb{P}\bigl(S(\check{Y})=k\bigr) &=& \mathbb{E} \Biggl[\sum
_{m=k}^{\infty} \frac{\tilde{B}^m}{m!}\mathrm{e}^{-\tilde{B}}
\frac{m!}{(m-k)!} (v_k)^{m-k} G_k(1|V) \Biggr]
\\
&=& \mathbb{E} \bigl[{\tilde{B}}^k \mathrm{e}^{-\tilde
{B}(1-v_k)}G_k(1|V)
\bigr],
\end{eqnarray*}
whence
%
%
\begin{equation}
\label{suslocal} \mathbb{E}_{\check{Y}}\bigl[f_{S(\check{Y})|\check{Y}}(s)\bigr]= \sum
_{k=0}^{\infty} (-s)^k
\phi_{\tilde{B}}^{(k)}(1-v_k) G_k(1|V).
\end{equation}
\end{appendix}
\section*{Acknowledgements}
We thank the referees for their careful reading of the manuscript and
constructive comments which have improved considerably the presentation
of the paper.


%



\printaddresses

\end{document}